\documentclass{aic}

\usepackage{pgf,tikz}
\usetikzlibrary{arrows}
\usepackage[center]{caption}

\usepackage{pgf,tikz}
\usetikzlibrary{arrows,matrix}

\numberwithin{equation}{section}

\theoremstyle{plain}
\newtheorem{thm}{Theorem}[section]
\newtheorem{lemma}[thm]{Lemma}
\newtheorem{prop}[thm]{Proposition}
\newtheorem{coroll}[thm]{Corollary}

\newtheorem{fact}[thm]{Fact}

\theoremstyle{definition}
\newtheorem{defn}[thm]{Definition}
\newtheorem{bij}[thm]{Bijection}

\newtheorem{remark}[thm]{Remark}
\newtheorem{ex}[thm]{Example}
\newtheorem{prob}[thm]{Problem}

\newcommand{\Arb}{{\rm Arb}}
\newcommand{\cE}{{\mathcal{E}}}
\newcommand{\med}{{\rm med}}
\newcommand{\nex}{{\rm next}}
\newcommand{\nextout}{{\rm nextout}}
\newcommand{\prevout}{{\rm prevout}}

\definecolor{v}{rgb}{0.25,0.05,0.7}

\aicAUTHORdetails{%
  title = {On Canonical Sandpile Actions of Embedded Graphs}, %% please capitalize all significant words
  author = {Lilla T\'othm\'er\'esz},
  plaintextauthor = {Lilla T\'othm\'er\'esz},
  keywords = {Sandpile group, rotor-routing, embedding, basepoint-independence, quasi-tree, medial digraph},
}   %%% END \aicAUTHORdetails

\aicEDITORdetails{%
   year={2026},
   number={6},
   received={24 June 2025},   % received date: example: 7 January 2017
   published={10 September 2026},  % published date
   doi={10.19086/aic.2026.6},      % XXX = number of paper, e.g. aic006 for paper#6
}   %%% END \aicEDITORdetails

\begin{document}

\begin{frontmatter}[classification=text]
%% EDITOR: this will force the keywords to appear right after the Abstract.
%%   If the abstract is too long and would force the keywords off the
%%   front page, please comment out % [classification=text] above
%%   This way the keywords will be floated on the bottom of the first page
%%   even though the Abstract spills over to the next page.

%%% AUTHOR: Title goes here.  This line is optional.  You must use it
%%   if title has footnote attached or requires nontrivial typesetting,
%%   e.g., inclusion of linebreaks to force nice layout.
\title{On Canonical Sandpile Actions of Embedded Graphs}%\titlefootnote{This is a footnote to the title}} %% please capitalize all significant words

%%% AUTHOR:
%%% List all authors. If you wish, place grant acknowledgements in \thanks.
%%% In brackets include a short tag for each author.
\author[pgom]{Lilla T\'othm\'er\'esz\thanks{This work was supported by the Counting in Sparse Graphs Lendület Research Group of the Alfr\'ed Rényi Institute of Mathematics.
}}

%%% AUTHOR: Abstract goes here
\begin{abstract}
    The sandpile group of a connected graph is a group whose cardinality is the number of spanning trees. The group is known to have a canonical simply transitive action on spanning trees if the graph is embedded into the plane. However, no canonical action on the spanning trees was known for the nonplanar case. 
We show that for any embedded Eulerian digraph, one can define a canonical simply transitive action of the sandpile group on compatible Eulerian tours (a set whose cardinality equals the number of spanning arborescences). This enables us to give a new proof that the rotor-routing action of a ribbon graph is independent of the root if and only if the embedding is into the plane (originally proved by Chan, Church and Grochow).

Recently, Merino, Moffatt and Noble defined a sandpile group variant (called Jacobian) for embedded graphs, whose cardinality is the number of quasi-trees. Baker, Ding and Kim showed that this group acts canonically on the quasi-trees.
We show that the Jacobian of an embedded graph is the usual sandpile group of the medial digraph, and the action by Baker et al.~agrees with the action of the sandpile group of the medial digraph on Eulerian tours (which fact is made possible by the existence of a canonical bijection between Eulerian tours of the medial digraph and quasi-trees due to Bouchet).
\end{abstract}
\end{frontmatter}

\section{Introduction}

The sandpile group of a connected graph is an Abelian group, whose order equals the number of spanning trees. More generally, the sandpile group of an Eulerian digraph is a group whose order equals the number of in-arborescences rooted at an arbitrary vertex. Ellenberg \cite{Ellenberg} asked if the sandpile group of an (undirected) graph has a ``canonical'' simply transitive action on the spanning trees. More precisely, he asked if such a canonical action exists,  if the graph is embedded into an orientable surface. Holroyd,  Levine, Mészáros, Peres, Propp and Wilson \cite{Holroyd08} defined the rotor routing action, which is an action of the sandpile group on the spanning trees, and depends on an embedding and a fixed vertex (root) of the graph. 
Chan, Church and Grochow \cite{Chan15} showed that the rotor-routing action is independent of the root if and only if the embedding is into the plane. Analogous results were proved about another sandpile group action associated to embedded graphs, the Bernardi action \cite{Baker-Wang}, and the relationship of the two sandpile group actions also has a rich literature \cite{Baker-Wang,ding2021rotor,SW_torsor,trinity}, in particular, the rotor-routing action agrees with the Bernardi action for plane graphs, but not in general. Kálmán, Lee, and the author of this paper gave a basepoint-free definition for the planar rotor-routing/Bernardi action \cite{trinity}. Furthermore, Ganguly and McDonough proved there is essentially a unique sandpile action for plane graphs, if one requires nice behavior with respect to deletion and contraction \cite{GM}. We note that the rotor routing action is also defined for sandpile groups of digraphs, as an action on spanning in-arborescences. 

The existence of a canonical action in the non-plane case remained open for a while. Recently, Merino, Moffatt and Noble \cite{merino2023critical} defined an ``embedded sandpile group'' called the Jacobian\footnote{The term ``Jacobian'' is typically used synonymously with ``sandpile group''. In this paper, we will use Jacobian for the embedded sandpile group of Merino, Moffatt and Noble, and sandpile group for the abstract sandpile group.} of the embedded graph. This group has cardinality equal the number of quasi-trees. Quasi-trees are those spanning subgraphs of an embedded graph that have one boundary component. If the embedding is into the plane, they are exactly the spanning trees, but for higher genus, there are quasi-trees that are not trees. The Jacobian is isomorphic to the sandpile group if the embedding is into the plane.
Baker, Ding and Kim \cite{BDK} gave an alternative definition for the Jacobian group of an embedded graph (and generalized it to regular orthogonal matroids), and they proved that the Jacobian of an embedded graph has a canonical action on the quasi-trees.

We have three main results: 
\begin{enumerate}
	%\paragraph{1.} 
	\item Firstly, we show that, in fact, any embedded Eulerian digraph has a canonical sandpile action if instead of spanning trees, one considers Eulerian tours compatible with the embedding. We say that an Eulerian tour is compatible with the embedding if at each vertex $v$, the order in which the tour traverses the out-edges of $v$ agrees with the positive orientation around $v$. We call the action the tour-rotor action, since it can be obtained from rotor routing. 
	More precisely:
	For an embedded Eulerian digraph, with any root vertex, the rotor-routing action of the sandpile group can be lifted to a canonical action on Eulerian tours compatible with the embedding.  Interestingly, we do not yet see a way to give a canonical definition for the tour-rotor action, though we suspect that there should be such a definition. These results can be found in Section \ref{sec:tour-rotor_action}.
	
	%\paragraph{2.} 
	\item Based on the tour-rotor action, we give a new proof that the rotor-routing action of an embedded graph is independent of the root if and only if the embedding is into the plane (originally due to Chan, Church and Grochow \cite{Chan15}). Although the main ideas behind the new proof are very similar to those of \cite{Chan15}, we feel that the new proof is technically a bit simpler. Also, using the tour-rotor action, one can formulate a definition for the rotor-routing action which is similar to the definition of the Bernardi action. This might help to find further connections between these actions. These results can be found in Section \ref{s:action_of_bidirected_graphs}.
	
	%\paragraph{3.}
	\item Finally, we show that the Jacobian of an embedded graph is canonically isomorphic to the sandpile group of the medial digraph. It is known \cite{Bouchet} that the Eulerian tours of the medial digraph are canonically in bijection with the quasi-trees. Furthermore, in this case, all Eulerian tours are compatible. We show that the action defined by Baker, Ding and Kim \cite{BDK} agrees with the tour-rotor action of the sandpile group of the medial digraph. 
	These results can be found in Section \ref{s:medial}.
\end{enumerate}

Previously, Kálmán, Lee and the author proved \cite{trinity} that for a plane embedded graph $G$, the sandpile group of $G$ is canonically isomorphic to the sandpile group of the medial digraph. In some sense, (3) generalizes this result. The paper \cite{trinity} also proved that the sandpile group of the medial digraph of a plane graph has a canonical simply transitive action on spanning trees. Indeed, one can think of (characteristic vectors of) spanning trees as chip-configurations for the medial digraph. In \cite{trinity}, it was proved that spanning trees represent the linear equivalence classes with $|V|-1$ chips for the medial digraph, hence the sandpile group of the medial digraph acts on them by addition. The Bernardi action agrees with this action via the isomorphism of the sandpile group and the medial sandpile group.

If the embedding is not into the plane, both elements of the above picture break down: The sandpile group of the medial digraph is no more isomorphic to the sandpile group of $G$, and spanning trees are no longer a system of representatives of a coset of the medial sandpile group.
However, as it turns out, the medial sandpile group still has a canonical action on the quasi-trees.

\section{Preliminaries}\label{sec:prelim}

\subsection{Graphs, spanning trees and spanning arborescences}

In this paper we denote undirected graphs by $G$, while we denote directed graphs by $D$. We allow multiple edges and loops. We will always suppose  that our undirected graphs are connected and out directed graphs are weakly connected. 
%By a slight abuse of notation, we will often use the notation $\overrightarrow{uv}$ for an edge 
In an undirected graph, we denote the degree of $v$ by $d(v)$. In a directed graph, we denote the out-degree of $v$ by $d^+(v)$ and the in-degree of $v$ by $d^-(v)$. We use the notation $\overrightarrow{uv}_e$ for an edge $e$ whose tail is $u$ and head is $v$.

For an undirected graph, a \emph{spanning tree} is a connected, cycle-free subgraph containing each vertex.
In a directed graph, a \emph{spanning in-arborescence} rooted at $v$ is a subgraph whose underlying undirected graph is a spanning tree, and there is a directed path from each vertex to $v$. We denote the set of spanning in-arborescences of a digraph $D$ rooted at $v$ by $\Arb(D,v)$.
A spanning in-arborescence $T\in\Arb(D,v)$ has out-degree 1 for each vertex $u\neq v$, and outdegree 0 for $v$. We use the notation $T(u)$ for the (unique) edge leaving node $u\neq v$.

\subsection{Ribbon graphs and Eulerian tours}
A \emph{ribbon digraph} $D$ is a digraph together with a cyclic ordering of the (in- and out-) edges incident to $v$ for each vertex $v$. For a vertex $v$, and incident edge $e$, we denote the edge following $e$ at $v$ by $\nex(v,e)$. Here, the orientation of $e$ can be anything and also for $\nex(v,e)$.

It is well known, that for each ribbon digraph $D$, one can construct a closed orientable surface $\Sigma$ such that $D$ is cellularly embedded in $\Sigma$, and the positive orientation around each vertex of $D$ induces the ribbon structure (see for example \cite{Thomassen_embeddings}). We will alternatively use the abstract or the embedding viewpoint.

A digraph $D$ is \emph{Eulerian} if the in-degree agrees with the out-degree for each vertex. An Eulerian ribbon digraph is \emph{balanced}, if around each vertex, the in- and the out-edges alternate in the cyclic order. 

An (undirected) \emph{ribbon graph} $G$ is a graph together with a cyclic ordering of the edges incident to $v$ for each vertex $v$. 
To an undirected ribbon graph $G$, we can naturally associate a balanced Eulerian ribbon digraph $G^{\leftrightarrows}$: 
Substitute each edge $e=uv$ with two oppositely directed edges $\overrightarrow{uv}_{e_1}$ and $\overrightarrow{vu}_{e_2}$, let $\overrightarrow{vu}_{e_2}=\nex(u, \overrightarrow{uv}_{e_1})$ and $\overrightarrow{uv}_{e_1}=\nex(v, \overrightarrow{vu}_{e_2})$, and otherwise insert the two edges to the original place of $e$ in the ribbon structure of $G$. We will call $G^{\leftrightarrows}$ the \emph{normal bidirected ribbon graph} corresponding to $G$. Note that if the ribbon structure of $G$ was obtained from the embedding into an orientable surface $\Sigma$, then the ribbon structure of $G^{\leftrightarrows}$ can also be obtained from $\Sigma$ and vice versa. In particular, $G$ is a planar ribbon graph if and only if $G^{\leftrightarrows}$ is planar. 

An \emph{Eulerian tour} of a digraph $D$ is a closed walk that uses each edge of $D$ exactly once. It is well-known that a digraph has an Eulerian tour if and only if it is Eulerian. 
We will think of an Eulerian tour without specifying its beginning and endpoint. That is, for us, an Eulerian tour is a cyclic ordering of the edge set such that for any edge, its head agrees with the tail of its subsequent edge.

For an Eulerian ribbon digraph $D$, we call an Eulerian tour of $D$ \emph{compatible with the ribbon structure}, if around each vertex $v$, the cyclic order in which the Eulerian tour traverses the out-edges of $v$ agrees with the cyclic order induced by the ribbon structure. Note that we do not assume anything about the order in which the tour traverses the in-edges of $v$. 

\begin{ex}
	Figure \ref{f:compatible_tour_and_arborescence} shows an Eulerian ribbon digraph embedded into the plane. The red curve indicates a compatible Eulerian tour, that corresponds to the cyclic ordering $e_1, e_4, e_3, e_2, e_8, e_9, e_5, e_6, e_7$ of the edges. To check that this is a compatible tour, note that around the middle vertex, the out-edges are traversed in the cyclic order $e_4,e_2,e_6$ which corresponds to the positive orientation. The other vertices have out-degree 2, hence compatibility is automatically satisfied at those vertices.
\end{ex}

\begin{figure}
	\begin{center}
		\begin{tikzpicture}[-,>=stealth',auto,scale=0.8,ultra thick]
			\tikzstyle{v}=[circle,scale=0.5,fill,draw]
			\begin{scope}[shift={(-4,0)}]
				\node[v] (1) at (0, 0) {};
				\node[v] (2) at (4, 0) {};
				\node[v] (3) at (2, 3.4) {};
				\node[v] (4) at (2, 1.2) {};
				\draw[color=red,rounded corners=10pt,line width=3pt] (4) -- (2.2, 2.3) -- (2, 3.1) -- (1.7, 2.3) -- (1.9, 1.3) -- (0.8, 0.8) -- (0.2, 0.4) -- (0.6, 1.8) -- (2, 3.5) -- (3.4, 1.8) -- (3.8, 0.4) -- (3.2, 0.8) -- (2.2, 1.1)  -- (2.8, 0.4)-- (3.7, 0) -- (2, -0.4) -- (0.3, 0) -- (1.2, 0.4) -- (4);
				%\path[line width=1.5pt,->]
				%	(1) edge [bend right=20] node {} (4)
				%	(4) edge [bend right=20] node {} (1)
				%	(2) edge [bend right=20] node {} (4)
				%	(4) edge [bend right=20] node {} (2);
				\draw[->,rounded corners=15pt] (1) -- (1.2, 0.4) -- (4);
				\draw[->,rounded corners=15pt] (4) -- (0.8, 0.8) -- (1);
				\draw[->,rounded corners=15pt] (4) -- (2.8, 0.4) -- (2);
				\draw[->,rounded corners=15pt] (2) -- (3.2, 0.8) -- (4);
				\draw[->,rounded corners=15pt] (3) -- (1.7, 2.3) -- (4);
				\draw[->,rounded corners=15pt] (4) -- (2.2, 2.3) -- (3);
				\draw[->,rounded corners=20pt] (3) -- (3.4, 1.8) -- (2);
				\draw[->,rounded corners=20pt] (2) -- (2, -0.4) -- (1);
				\draw[->,rounded corners=20pt] (1) -- (0.6, 1.8) -- (3);
				\node at (1.4,0.3) {$e_1$};
				\node at (0.85, 1.15) {$e_2$};
				\node at (1.5,1.9) {$e_3$};
				\node at (2.4,1.9) {$e_4$};
				\node at (2.8,1.1) {$e_5$};
				\node at (2.7,0.3) {$e_6$};
				\node at (2,-0.7) {$e_7$};
				\node at (0.6, 2.5) {$e_8$};
				\node at (3.4, 2.5) {$e_9$};
				\node at (-0.2, 3.2) {$\cE$};
			\end{scope}
			\begin{scope}[shift={(3,0)}]
				\node[v] (1) at (0, 0) {};
				\node[v] (2) at (4, 0) {};
				\node[v] (3) at (2, 3.4) {};
				\node[v] (4) at (2, 1.2) {};
				%\draw[color=red,rounded corners=10pt,line width=3pt] (4) -- (2.2, 2.3) -- (2, 3.1) -- (1.7, 2.3) -- (1.9, 1.3) -- (0.8, 0.8) -- (0.2, 0.4) -- (0.6, 1.8) -- (2, 3.5) -- (3.4, 1.8) -- (3.8, 0.4) -- (3.2, 0.8) -- (2.2, 1.1)  -- (2.8, 0.4)-- (3.7, 0) -- (2, -0.4) -- (0.3, 0) -- (1.2, 0.4) -- (4);
				\draw[->,dashed,rounded corners=15pt] (1) -- (1.2, 0.4) -- (4);
				\draw[->,dashed,rounded corners=15pt] (4) -- (0.8, 0.8) -- (1);
				\draw[->,rounded corners=15pt,line width=2pt] (4) -- (2.8, 0.4) -- (2);
				\draw[->,dashed,rounded corners=15pt] (2) -- (3.2, 0.8) -- (4);
				\draw[->,dashed,rounded corners=15pt] (3) -- (1.7, 2.3) -- (4);
				\draw[->,dashed,rounded corners=15pt] (4) -- (2.2, 2.3) -- (3);
				\draw[->,rounded corners=20pt,line width=2pt] (3) -- (3.4, 1.8) -- (2);
				\draw[->,rounded corners=20pt,line width=2pt] (2) -- (2, -0.4) -- (1);
				\draw[->,dashed,rounded corners=20pt] (1) -- (0.6, 1.8) -- (3);
				\node at (1.4,0.3) {$e_1$};
				\node at (0.85, 1.15) {$e_2$};
				\node at (1.5,1.9) {$e_3$};
				\node at (2.4,1.9) {$e_4$};
				\node at (2.8,1.1) {$e_5$};
				\node at (2.7,0.3) {$e_6$};
				\node at (2,-0.7) {$e_7$};
				\node at (0.6, 2.5) {$e_8$};
				\node at (3.4, 2.5) {$e_9$};        
				\node at (-0.2, 3.2) {$A_{e_1}(\cE)$};
				\draw[->] (-2, 1.5) -- (-1, 1.5);
			\end{scope}
		\end{tikzpicture}
	\end{center}
	\caption{An Eulerian tour $\cE$ compatible with the planar embedding (left panel). It corresponds to the cyclic ordering $e_1, e_4, e_3, e_2, e_8, e_9, e_5, e_6, e_7$. The right panel shows the in-arborescence $A_{e_1}(\cE)$ obtained from $\cE$ by choosing $e_1$ as the initial edge of the tour. \label{f:compatible_tour_and_arborescence}}
\end{figure}
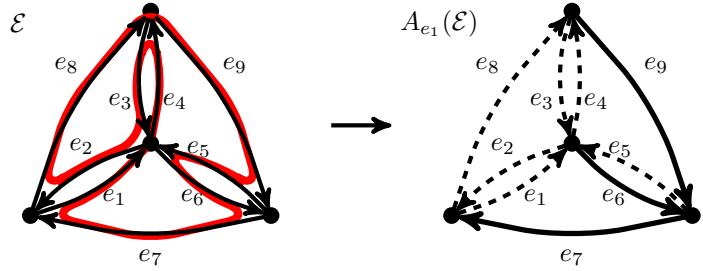

By the BEST theorem \cite{de1951circuits}, for any given vertex $v$, the number of compatible Eulerian tours of $D$ agrees with the number of in-arborescences of $D$ rooted at $v$.
Let us explain a bijection between these objects, as we will need it.

\begin{defn}[bijection between compatible Eulerian tours and $\Arb(D,v)$, \cite{de1951circuits}]\label{def:Euler_tour_arb_bijection}
	To define the bijection, we need to fix an edge $e=\overrightarrow{vw}_e$. %The bijection is between compatible Eulerian tours, and in-arborescences rooted at $v$.
	For a compatible Eulerian tour $\mathcal{E}$, let us now think of $\overrightarrow{vw}_e$ as its first edge. Then, for each vertex $u\neq v$, take the out-edge of $u$ that is last used by the tour. This way, we get a subgraph where the out-degree of each vertex is 1, except for $v$, where the out-degree is 0. By \cite{de1951circuits}, this subgraph is an in-arborescence rooted at $v$. 
	We denote the obtained in-arborescence by $A_{\overrightarrow{vw}_e}(\mathcal{E})$.
\end{defn}

The above assignment is a bijection between the set of compatible Eulerian tours, and in-arborescences rooted at $v$ \cite{de1951circuits}. See Figure \ref{f:compatible_tour_and_arborescence} for an example. We use $A^{-1}_{\overrightarrow{vw}_e}$ for the inverse map.
Note that in \cite{de1951circuits} they prove that there are $\prod_{v\in V} (d^+(v)-1)!$ Eulerian tours where the above procedure gives $A_{\overrightarrow{vw}_e}(\mathcal{E})$. These correspond exactly to the ways to choose the cyclic order in which the out-edges are traversed around each vertex. If these cyclic orders are fixed (as is the case with ribbon structures and compatible Eulerian tours) we get a bijection.

\subsection{Sandpile group}

Let $D=(V,E)$ be an Eulerian digraph. 

We denote the free Abelian group on $V$ by $\mathbb{Z}^V$. We refer to a vector $\mathbf{x}\in \mathbb{Z}^V$ as a chip configuration. We denote the coordinate of $\mathbf{x}$ corresponding to $v\in V$ by $\mathbf{x}(v)$. We think of $\mathbf{x}(v)$ as the number of chips on $v$ (which might also be negative). %We use the notation $\deg(x)=\sum_{v\in V} x(v)$, and call $\deg(x)$ the \emph{degree} of $x$.
We denote by $\mathbb{Z}_0^V$ the subgroup of $\mathbb{Z}^V$ where the sum of the coordinates is 0.

The sandpile group is defined via the so-called chip-firing moves. The \emph{firing} of a vertex $v$ modifies chip configuration $\mathbf{x}$ by decreasing the number of chips on $v$ by the outdegree of $v$, and increasing the number of chips on each $w$ by the number of edges pointing from $v$ to $w$. 
We define two chip configurations $\mathbf{x}$ and $\mathbf{y}$ to be \emph{linearly equivalent} if there is a sequence of firings that transforms $\mathbf{x}$ to $\mathbf{y}$. We denote $\mathbf{x}\sim_D \mathbf{y}$. If the graph $D$ is clear from the context, we simply write $\mathbf{x} \sim \mathbf{y}$. Linear equivalence is indeed an equivalence relation (symmetry is the only nontrivial property, but since $D$ is Eulerian, firing all vertices but $v$ reverses the operation of firing $v$). Linear equivalence preserves the sum of the chips. Note that the addition or removal of loop edges to $D$ does not modify the notion of linear equivalence, since the effect of a loop edge ``cancels out'' during a firing.

\begin{defn}[Sandpile group]
	For an Eulerian digraph $D=(V,E)$, the \emph{sandpile group} is defined as $S(D)=\mathbb{Z}_0^V/_{\sim_D}$.
\end{defn}

The following statement is a version of the matrix-tree theorem. % \cite{stanley}.

\begin{fact}\cite{Holroyd08}\label{fact:order_of_sandpile_group}
	For an Eulerian digraph $D$, the order of $S(D)$ is equal to the number of arborescences rooted at an arbitrary vertex.
\end{fact}

For a nonloop edge $e=\overrightarrow{vw}_e$, let us denote by $\chi_e\in\mathbb{Z}^V$ the vector that has 
\begin{equation}\label{eq:chi}
	\chi_e(u) = \left\{\begin{array}{cl} 1 & \text{if $u=w$,}  \\
		-1 & \text{if $u=v$},\\
		0 & \text{otherwise}.
	\end{array} \right.
\end{equation}

If $e$ is a loop (that is, $v=w$) then $\chi_e$ is defined as $\mathbf{0}$.
Note that for a strongly connected digraph $D$, the vectors $\{\chi_e: e\in E(D)\}$ generate $\mathbb{Z}_0^V$, hence their equivalence classes generate the sandpile group $S(D)$.

The sandpile group of an undirected graph $G$ is defined as $S(G):=S(G^{\leftrightarrows})$, where by $G^{\leftrightarrows}$ we mean the bidirected Eulerian digraph where each edge $uv$ is substituted with two oppositely directed edges $\overrightarrow{uv}$ and $\overrightarrow{vu}$.

In Section \ref{s:action_of_bidirected_graphs}, we will need a condition to show that a chip configuration $\mathbf{x}$ on a bidirected graph $G^{\leftrightarrows}$ is not linearly equivalent to $\mathbf{0}$ (the all-zero vector). The following condition appeared for example in \cite{Baker-Wang,BBY} but as the terminologies are a bit different, we give a proof.

Let $G^{\leftrightarrows}$ be a bidirected graph.
Let us call an edge set $C\subseteq E(G^{\leftrightarrows})$ a directed cycle if it spans a connected subgraph and each vertex in $V(C)$ has in-degree 1 and out-degree 1.
Let us call an edge set $C^*$ a directed elementary cut, if there exists a partition $V=U\sqcup W$ such that $C^*=\{\overrightarrow{uw}_e\in E(G^{\leftrightarrows}): u\in U, w\in W\}$ and $U$ and $W$ both span connected subgraphs in $G$. (Note that there will also be edges leading from $W$ to $U$, we do not include these into $C^*$.) We denote $C^*=C^*(U,W)$.

\begin{prop}\cite{Baker-Wang,BBY}\label{prop:zero_lin_ekv_cond}
	Suppose that $F\subseteq E(G^{\leftrightarrows})$ is an edge set. We have $\sum_{e\in F} \chi_e\sim\mathbf{0}$ if and only if $F$ is a disjoint union of directed cycles and directed elementary cuts.
\end{prop}
\begin{proof}
	Firing each vertex in $U$ once for some $U\subset V$ transforms $\mathbf{0}$ into $\sum_{e\in C^*(U,V-U)} \chi_e$. 
	Also, if $C$ is a directed cycle, then $\sum_{e\in C}\chi_e=\mathbf{0}$. This proves the ``if'' direction.

	For the ``only if'' direction, note that if we have $\sum_{e\in E(G^{\leftrightarrows})}\lambda_e \chi_e=\sum_{e\in E(G^{\leftrightarrows})}\mu_e \chi_e$, then $\lambda-\mu$ is a circulation. In particular, for each partition $V=U\sqcup W$, 
	we have $\sum\{\lambda_{\overrightarrow{uw}_e}-\mu_{\overrightarrow{uw}_e} : u\in U, w\in W\}=\sum\{\lambda_{\overrightarrow{wu}_f}-\mu_{\overrightarrow{wu}_f} : u\in U, w\in W\}$ or in other words
	\begin{equation}\label{eq:circulation}
		\sum\{\lambda_{\overrightarrow{uw}_e}-\lambda_{\overrightarrow{wu}_f} : u\in U, w\in W\}=\sum\{\mu_{\overrightarrow{uw}_e}-\mu_{\overrightarrow{wu}_f} : u\in U, w\in W\}.
	\end{equation}
	
	Suppose now that $\mathbf{x}:= \sum_{e\in F} \chi_e\sim\mathbf{0}$. Since for a directed cycle $C$, we $\sum_{e\in C}\chi_e=\mathbf{0}$, we can suppose that $F$ has no directed cycle as subgraph.
	
	Take the sequence of firings that transforms $\mathbf{0}$ to $\sum_{e\in F} \chi_e$. Since firing each vertex in $V$ once does not modify the chip configuration, we can suppose that each vertex is fired a nonnegative number of times, and there is at least 1 vertex that is fired 0 times. Let $V_i$ $(i=0,1,\dots)$ be the set of vertices that are fired $i$ times. Then, we can describe the net chip movement as follows: If for $\overrightarrow{uv}_e$, we have $u\in V_{i+j}$ and $v\in V_i$ for some $j\geq 0$, then $j$ chips traverse $\overrightarrow{uv}_e$. Take
	\begin{equation*}
		\mu_{\overrightarrow{uv}_e} = \left\{\begin{array}{cl} j & \text{if $u\in V_{i+j},$ $v\in V_i$ with $j\geq 1$}  \\
			0 & \text{if $u\in V_{i+j},$ $v\in V_i$ with $j \leq 0$} .
		\end{array} \right.
	\end{equation*}
	Then, $\sum_{e\in E(G^{\leftrightarrows})}\mu_e \chi_e = \sum_{e\in F} \chi_e$.
	
	We claim that we cannot have an edge $\overrightarrow{uv}_e$ leading from $V_{i+j}$ to $V_i$ with $j>1$. Indeed, if that is the case, then by taking $W:=V_0 \cup \dots \cup V_i$ and $U:=V_{i+1} \cup \dots$, each edge $e$ leading from $U$ to $W$ has $\mu_e\geq 1$, and at least one of them has $\mu_e> 1$, while each edge leading from $W$ to $U$ has $\mu_e=0$.
	On the other hand, in $\sum_{e\in F} \chi_e$, each edge has coefficient 0 or 1, hence we cannot satisfy \eqref{eq:circulation} for these two sets of coefficients.
	
	Suppose that $k$ is the largest number such that $V_k\neq\emptyset$. Consequently, $V_i\neq \emptyset$ for each $i=0, \dots k$, and $F=\cup_{i=1}^k C^*(V_0 \cup \dots \cup V_{i-1},V_i \cup \dots \cup V_k)$. This finishes the proof.
\end{proof}

\subsection{Rotor-routing}
\subsubsection{Eulerian digraphs}

Rotor routing is a relaxed version of chip-firing. It is also interesting since one can think of it as a derandomized random walk \cite{rr_and_markov}. From our point of view, the most interesting property of rotor routing is that it enables one to define a free, transitive action of the sandpile group of an Eulerian digraph $D$ on $\Arb(D,v)$.

Rotor-routing is defined for a ribbon digraph. For a directed edge $\overrightarrow{vu}_e$, let us introduce the notation $\nextout(v, \overrightarrow{vu}_e)$ for the next out-edge after $\overrightarrow{vu}_e$ in the cyclic order around $v$. Also, let us denote by $\prevout(v, \overrightarrow{vu}_e)$ the out-edge preceding $\overrightarrow{vu}_e$ in the cyclic order at $v$.

A \emph{rotor configuration} on $D$ is a function $\varrho$ that assigns to each 
vertex $v$ an out-edge with tail $v$. We call $\varrho(v)$ the \emph{rotor} at $v$.

A configuration of the rotor-routing game is a pair $(\mathbf{x},\varrho)$, where $\mathbf{x}$ is a chip-configuration, and $\varrho$ is a rotor configuration on $D$. We also call such pairs \emph{chip-and-rotor configuration}.

Given a configuration $(\mathbf{x},\varrho)$, a \emph{routing} at vertex $v$ results in the  
configuration $(\mathbf{x}', \varrho')$, where
$\varrho'$ is the rotor configuration with
$$
\varrho'(u) = \left\{\begin{array}{cl} \varrho(u) & \text{if $u\neq v$,}  \\
	\nextout(v,\varrho(v)) & \text{if $u=v$},
\end{array} \right.
$$
and $\mathbf{x}'=\mathbf{x}-\mathbf{1}_v+\mathbf{1}_{v'}$ where $v'$ is the head of $\varrho'(v)$.

We call the routing at $v$ \emph{legal} (with respect to the configuration $(\mathbf{x},\varrho)$), if $\mathbf{x}(v)>0$, i.e.~the routing at $v$ does not create a negative entry at $v$. Note that other vertices might have a negative number of chips. 
We talk about an \emph{unconstrained routing}, or simply a \emph{routing} if we do not require anything about $\mathbf{x}(v)$.

Let $D$ be an Eulerian ribbon digraph.
Holroyd et al.~\cite{Holroyd08} defined a free, transitive action of the sandpile group of $D$ on $\Arb(D,v)$, using the rotor-router operation. For completeness, we give the original definition, but we will instead work with an alternative characterization from \cite{alg_rotor} given after the definition.

\begin{defn}[Rotor-router action, \cite{Holroyd08}]\label{def::rr_action}
	Let $D$ be an Eulerian digraph. The \emph{rotor-router action} is defined with respect to a base vertex $v\in V$ that we call the \emph{root}. It is a group action of $S(D)$ on $\Arb(D,v)$, denoted by $r_v$.
	
	For some $\mathbf{x}\in S(D)$ and $T\in\Arb(D,v)$, the arborescence $r_v(\mathbf{x},T)$ is defined as follows: Choose a chip configuration $\mathbf{x}'\sim \mathbf{x}$ such that $\mathbf{x}'(u)\geq 0$ for each $u\neq v$. Such an $\mathbf{x}'$ can easily be seen to exist.
	Fix any out-edge $\overrightarrow{vw}_e$ of $v$.
	Let $\varrho=T\cup\overrightarrow{vw}_e$. (This is indeed a rotor configuration.)
	
	Start a legal rotor-routing game from $(\mathbf{x}',\varrho)$,
	such that $v$ is not allowed to be routed.
	Continue until each chip arrives at $v$. Holroyd et al.~\cite{Holroyd08} shows that this procedure ends after finitely many steps, and in the final configuration $(\mathbf{0},\varrho')$, the edges $\{\varrho'(u):u\in V-v\}$ form a spanning in-arborescence of $D$ rooted at $v$. $r_v(\mathbf{x},T)$ is defined to be this arborescence.
\end{defn}

Holroyd et al.~\cite{Holroyd08} shows that for Eulerian digraphs, $r_v(\mathbf{x},T)$ is well defined, i.e.,~the definition does not depend on our choice of $\mathbf{x}'$ and on the choice of the legal game. (It is easy to see that the choice of $w$ is immaterial in the construction.) From this, it also follows that this is indeed a group action of $S(D)$ on $\Arb(D,v)$, i.e.,~$r_v(\mathbf{x},T)=r_v(\mathbf{x}',T)$ if $\mathbf{x}\sim_D \mathbf{x}'$. It is also known that $r_v$ is a simply transitive action \cite{Holroyd08}, that is, for each $T,T'\in \Arb(D,v)$, there is up to linear equivalence a unique chip-configuration $\mathbf{x}$ such that $r_v(\mathbf{x},T)=T'$.

As mentioned above, we will work with an alternative definition for the rotor-routing action from \cite{alg_rotor}, that we now explain.
This characterization uses a version of linear equivalence for rotor-routing:
\begin{defn}[linear equivalence of chip-and-rotor configurations, \cite{alg_rotor}]
	Two configurations $(\mathbf{x}_1,\varrho_1)$ and $(\mathbf{x}_2,\varrho_2)$ are linearly equivalent, if $(\mathbf{x}_2,\varrho_2)$ can be reached from $(\mathbf{x}_1,\varrho_1)$ by a sequence of unconstrained routings. We denote this by $(\mathbf{x}_1,\varrho_1)\sim_D (\mathbf{x}_2,\varrho_2)$, or by $(\mathbf{x}_1,\varrho_1)\sim (\mathbf{x}_2,\varrho_2)$ if the graph is clear from the context.
\end{defn}

If $D$ is strongly connected, this is indeed an equivalence relation \cite[Proposition 3.4]{alg_rotor} (symmetry is the only nontrivial property).
Note the following simple, but important fact:
\begin{fact}\label{fact:chip_lin_ekv_and_rotor_lin_ekv}
	If $\mathbf{x}$ and $\mathbf{y}$ are chip configurations, and $\varrho$ is an arbitrary rotor configuration of the strongly connected ribbon digraph $D$, then $(\mathbf{x},\varrho) \sim_D (\mathbf{y}, \varrho)$ if and only if $\mathbf{x}\sim_D \mathbf{y}$.
\end{fact}
\begin{proof}
	If the rotor configurations agree in some chip-and-rotor configurations, then in any routings transforming one to the other, each vertex $v$ needs to be routed $t\cdot d^+(v)$ times for some integer $t$, which has the same effect on the chip configuration as $t$ firings of $v$.
	Conversely, a firing of $v$ can be realized by routing vertex $v$ exactly $d^+(v)$ times, and the rotors return to their initial positions.
\end{proof}

\begin{prop}{\cite[Proposition 3.16]{alg_rotor}}
	\label{prop:alt_def_rotor_action}
	Let $D$ be an Eulerian ribbon digraph, $\mathbf{x}\in S(D)$, and $T,T'\in \Arb(D,v)$.
	Fix an arbitrary out-edge $\overrightarrow{vw}_e\in E$ of $v$.
	We have $r_v(\mathbf{x},T)=T'$ if and only if $(\mathbf{x}, T \cup \overrightarrow{vw}_e)\sim (\mathbf{0},T'\cup \overrightarrow{vw}_e)$.
\end{prop}

We note that \cite[Proposition 3.16]{alg_rotor} is phrased a little differently, but the above form is used for example in the proof of \cite[Theorem 3.17]{alg_rotor}.

\subsubsection{Undirected graphs}

The rotor-routing action is also defined for undirected ribbon graphs $G$, as an action of $S(G)$ on spanning trees. This is related to the Eulerian case as follows:

As explained in Section \ref{sec:prelim}, to an undirected ribbon graph $G$, one can associate the normal bidirected ribbon digraph $G^{\leftrightarrows}$.
By definition $S(G)=S(G^{\leftrightarrows})$. 

Also, for any fixed $u\in V$, there is a bijection between spanning trees of $G$ and $\Arb(G^{\leftrightarrows},u)$.
Let $T$ be a spanning tree of $G$. There is a unique orientation of $T$ that is an in-arborescence rooted at $u$. Let us denote by $T^u$ the in-arborescence of $G^{\leftrightarrows}$ corresponding to this orientation. Also, to any in-arborescence in $\Arb(G^{\leftrightarrows},u)$, we can associate the spanning tree of $G$ that we get by forgetting the orientations.

Now for some $\mathbf{x}\in S(G)(=S(G^{\leftrightarrows}))$ and spanning tre $T$, the rotor routing action of $G$ is defined as follows: 
\begin{defn}[Rotor-routing action for undirected graphs \cite{Holroyd08}]\label{def:rr_action_undirected}
	Let $r_u(\mathbf{x},T)$ be the spanning tree corresponding to the arborescence $r_u(\mathbf{x},T^u)$, where by $r_u(\mathbf{x},T^u)$ we mean the rotor-routing action of $S(G^{\leftrightarrows})$ on $\Arb(G^{\leftrightarrows},u)$.    
\end{defn}

\section{A canonical action of Eulerian ribbon digraphs on compatible Eulerian tours}\label{sec:tour-rotor_action}

Via the bijection between compatible Eulerian tours and $\Arb(D,v)$ (Definition \ref{def:Euler_tour_arb_bijection}), for each root vertex $v$, the rotor-router action with root $v$ can be lifted to an action on compatible Eulerian tours. 
Our main theorem is that this lifted action is canonical. That is, it does not depend on the root $v$, and neither on the edge $\overrightarrow{vw}_e$ used to define the bijection.

\begin{thm}\label{thm:Eulerian_tour_action_well_defined}
	Let $D$ be an Eulerian ribbon digraph, and $\overrightarrow{uv}_e$ and $\overrightarrow{wz}_f$ be arbitrary edges. 
	Let $\mathbf{x}\in S(D)$, and let $\mathcal{E}$ and $\mathcal{E}'$ be compatible Eulerian tours. If we have $$r_u(\mathbf{x}, A_{\overrightarrow{uv}_e}(\mathcal{E}))=A_{\overrightarrow{uv}_e}(\mathcal{E}'),$$ then
	$$r_w(\mathbf{x}, A_{\overrightarrow{wz}_f}(\mathcal{E}))=A_{\overrightarrow{wz}_f}(\mathcal{E}').$$
\end{thm}

We prove the Theorem at the end of this section.
As a corollary, one obtains that the following action of the sandpile group (of an Eulerian ribbon digraph) on compatible Eulerian tours is a well-defined, canonical action.
\begin{defn}[Tour-rotor action]\label{def:canonical_action_on_tours}
	Let $D$ be an Eulerian ribbon digraph. For $\mathbf{x}\in S(G)$ and a compatible Eulerian tour $\mathcal{E}$, we define $r(\mathbf{x},\mathcal{E})=\mathcal{E}'$, where $\mathcal{E'}$ is the unique compatible Eulerian tour such that for an arbitrary edge $\overrightarrow{vw}_e$, we have $r_v(\mathbf{x}, A_{\overrightarrow{vw}_e}(\mathcal{E}))=A_{\overrightarrow{vw}_e}(\mathcal{E}')$. We call this the tour-rotor action.
\end{defn}
\begin{coroll}\label{cor:Euler_sandpile_acts_on_tours}
	The tour-rotor action is a well-defined, canonical simply transitive action of the sandpile group of an Eulerian ribbon digraph on the set of compatible Eulerian tours.
\end{coroll}
\begin{ex}\label{ex:tour_rotor_action}
	Figure \ref{f:tour_rotor_action} shows a plane embedded Eulerian digraph. The left panel shows a chip configuration (white numbers on the vertices) and a compatible Eulerian tour (red curve). The tour-rotor action of the chip configuration on the tour gives the compatible Eulerian tour on the right panel (red curve).
\end{ex}

Before proving Theorem \ref{thm:Eulerian_tour_action_well_defined}, let us pose an open problem. One feels that this canonical action needs to have a nice, canonical definition, that does not need fixing an arbitrary edge. However, we were so far not able to come up with such a definition.

\begin{prob}\label{prob:canonical_def_for_tour_rotor_action}
	Give a canonical definition for the action defined in Definition \ref{def:canonical_action_on_tours}.
\end{prob}

The following lemma will be the key to proving Theorem \ref{thm:Eulerian_tour_action_well_defined}.

\begin{lemma}\label{lem:change_of_arborescences_when_changing_first_edge}
	Let $D$ be an Eulerian ribbon digraph, $\mathcal{E}$ a compatible Eulerian tour of $D$, and $\overrightarrow{uv}_e$ and $\overrightarrow{wz}_f$ two arbitrary edges of $D$. (Vertices in $\{u,v,w,z\}$ might coincide.) Then
	$$(\mathbf{1}_v-\mathbf{1}_{z}, A_{\overrightarrow{uv}_e}(\cE)\cup \overrightarrow{uv}_e) \sim (\mathbf{0}, A_{\overrightarrow{wz}_f}(\cE)\cup \overrightarrow{wz}_f).$$
\end{lemma}
\begin{proof} 
	First, we prove the statement of the Lemma if $\overrightarrow{uv}_e$ and $\overrightarrow{wz}_f$ are two consecutive edges in the Eulerian tour $\mathcal{E}$. That is, we take $\overrightarrow{uv}_e$ and $\overrightarrow{vz}_f$ such that $\overrightarrow{vz}_f$ is traversed right after $\overrightarrow{uv}_e$ in $\mathcal{E}$.
	
	When defining $A_{\overrightarrow{uv}_e}(\cE)$, we take $\overrightarrow{uv}_e$ as the first edge of $\cE$, and when defining $A_{\overrightarrow{vz}_f}(\cE)$, we take $\overrightarrow{vz}_f$ as the first edge, which means that in this case $\overrightarrow{uv}_e$ is the last edge.
	
	For any vertex different from $u$ and $v$, no matter whether we took $\overrightarrow{uv}_e$ or $\overrightarrow{vz}_f$ as the first edge, the first edge through which we reach the vertex does not change. Hence the last out-edge is the same for $w\in V-\{u,v\}$ in the two cases.
	
	For vertex $u$, the subgraph $A_{\overrightarrow{uv}_e}(\mathcal{E}) \cup \overrightarrow{uv}_e$ has out-edge $\overrightarrow{uv}_e$ by definition. Also, with $\overrightarrow{vz}_f$ as first edge, $\overrightarrow{uv}_e$ is the very last edge, hence it is also the last out-edge at $u$. Thus, $A_{\overrightarrow{vz}_f}(\cE) \cup \overrightarrow{vz}_f$ also has $\overrightarrow{uv}_e$ as the out-edge of $u$.
	
	For vertex $v$, $A_{\overrightarrow{vz}_f}(\cE) \cup \overrightarrow{vz}_f$ has out-edge $\overrightarrow{vz}_f$ by definition. On the other hand, with $\overrightarrow{uv}_e$ as first edge, $\overrightarrow{vz}_f$ is the first out-edge of $v$, hence $A_{\overrightarrow{uv}_e}(\cE) \cup \overrightarrow{uv}_e$ has $\prevout(v,\overrightarrow{vz}_f)$ as the out-edge of $v$.
	
	Altogether, the two rotor-configurations $A_{\overrightarrow{uv}_e}(\cE)$ and $A_{\overrightarrow{vz}_f}(\cE)$ agree at each vertex except for $v$, where the first one has out-edge $\prevout(v,\overrightarrow{vz}_f)$ and the latter one has $\overrightarrow{vz}_f$.
	For the chip-and-rotor configuration $(\mathbf{1}_v-\mathbf{1}_z, A_{\overrightarrow{uv}_e}(\mathcal{E}) \cup \overrightarrow{uv}_e)$, if we perform one routing step at $v$, we get exactly $(\mathbf{0}, A_{\overrightarrow{vz}_f}(\mathcal{E}) \cup \overrightarrow{vz}_f)$, hence the statement of the lemma is indeed true in the case if $\overrightarrow{uv}_e$ and $\overrightarrow{vz}_f$ are two consecutive edges of the Eulerian tour $\cE$.
	
	Now if $\overrightarrow{uv}_e$ and $\overrightarrow{wz}_f$ are arbitrary edges, take the sequence of edges between $\overrightarrow{uv}_e$ and $\overrightarrow{wz}_f$ on the tour $\cE$. That is, let $\overrightarrow{u_0u_1}_{e_1}=\overrightarrow{uv}_e, \overrightarrow{u_1u_2}_{e_2}, \dots, \overrightarrow{u_{k-1}u_k}_{e_{k}}=\overrightarrow{wz}_f$ be the sequence of edges of $\cE$ between $\overrightarrow{uv}_e$ and $\overrightarrow{wz}_f$.
	Applying the consecutive case successively, (and using the trivial fact that $(\mathbf{x},\varrho) \sim (\mathbf{x}', \varrho')$ implies $(\mathbf{x}+\mathbf{y},\varrho)\sim(\mathbf{x}'+\mathbf{y}, \varrho')$ for any $\mathbf{y}\in S(D)$), we get
	\begin{align*}
		((\mathbf{1}_{u_1}-\mathbf{1}_{u_2})+(\mathbf{1}_{u_2}-\mathbf{1}_{u_3})+\dots + (\mathbf{1}_{u_{k-1}}-\mathbf{1}_{u_{k}}), A_{\overrightarrow{u_0u_1}_{e_1}}(\cE)\cup \overrightarrow{u_0u_1}_{e_1})\\ \sim (\mathbf{0}, A_{\overrightarrow{u_{k-1}u_k}_{e_{k}}}(\cE)\cup \overrightarrow{u_{k-1}u_k}_{e_{k}}).
	\end{align*}
	After cancellation and substituting $u_1=v$ and $u_k=z$, we get
	$$(\mathbf{1}_v-\mathbf{1}_{z}, A_{\overrightarrow{uv}_e}(\cE)\cup \overrightarrow{uv}_e) \sim (\mathbf{0}, A_{\overrightarrow{wz}_f}(\cE)\cup \overrightarrow{wz}_f).$$
\end{proof}

\begin{figure}
	\begin{center}
		\begin{tikzpicture}[-,>=stealth',auto,scale=0.8,ultra thick]
			\tikzstyle{v}=[circle,scale=0.5,fill,draw]
			\begin{scope}[shift={(-4,0)}]
				\node[v] (1) at (0, 0) {\color{white}\Large -1};
				\node[v] (2) at (4, 0) {\color{white}\Large 0};
				\node[v] (3) at (2, 3.4) {\color{white}\Large -1};
				\node[v] (4) at (2, 1.2) {\color{white}\Large 2};
				\draw[color=red,rounded corners=10pt,line width=3pt] (4) -- (2.2, 2.3) -- (2, 3.1) -- (1.7, 2.3) -- (1.9, 1.3) -- (0.8, 0.8) -- (0.2, 0.4) -- (0.6, 1.8) -- (3) -- (3.4, 1.8) -- (3.8, 0.4) -- (3.2, 0.8) -- (2.2, 1.1)  -- (2.8, 0.4)-- (3.7, 0) -- (2, -0.4) -- (0.3, 0) -- (1.2, 0.4) -- (4);
				%\path[line width=1.5pt,->]
				%	(1) edge [bend right=20] node {} (4)
				%	(4) edge [bend right=20] node {} (1)
				%	(2) edge [bend right=20] node {} (4)
				%	(4) edge [bend right=20] node {} (2);
				\draw[->,rounded corners=15pt] (1) -- (1.2, 0.4) -- (4);
				\draw[->,rounded corners=15pt] (4) -- (0.8, 0.8) -- (1);
				\draw[->,rounded corners=15pt] (4) -- (2.8, 0.4) -- (2);
				\draw[->,rounded corners=15pt] (2) -- (3.2, 0.8) -- (4);
				\draw[->,rounded corners=15pt] (3) -- (1.7, 2.3) -- (4);
				\draw[->,rounded corners=15pt] (4) -- (2.2, 2.3) -- (3);
				\draw[->,rounded corners=20pt] (3) -- (3.4, 1.8) -- (2);
				\draw[->,rounded corners=20pt] (2) -- (2, -0.4) -- (1);
				\draw[->,rounded corners=20pt] (1) -- (0.6, 1.8) -- (3);
				%\node at (1.4,0.3) {$e_1$};
				%\node at (0.85, 1.15) {$e_2$};
				%\node at (1.5,1.9) {$e_3$};
				%\node at (2.4,1.9) {$e_4$};
				%\node at (2.8,1.1) {$e_5$};
				%\node at (2.7,0.3) {$e_6$};
				%\node at (2,-0.7) {$e_7$};
				%\node at (0.6, 2.5) {$e_8$};
				%\node at (3.4, 2.5) {$e_9$};
				%\node at (-0.2, 3.5) {$\cE$};
			\end{scope}
			\begin{scope}[shift={(3,0)}]
				\node[v] (1) at (0, 0) {\color{white} };
				\node[v] (2) at (4, 0) {\color{white} };
				\node[v] (3) at (2, 3.4) {\color{white} };
				\node[v] (4) at (2, 1.2) {\color{white} };
				\draw[color=red,rounded corners=10pt,line width=3pt] (2.15, 2) -- (2.2, 2.3) -- (2, 3.1) -- (1.7, 2.3) -- (1.9, 1.3) -- (0.8, 0.8) -- (0.2, 0.4) -- (0.6, 1.8) -- (3) -- (3.4, 1.8) -- (2) -- (2, -0.4) -- (0.3, 0) -- (1.2, 0.4) -- (2, 1) -- (2.8, 0.4) -- (3.8, 0.2) -- (3.2, 0.8) -- (2.1, 1.4) -- (2.15, 2);
				%\path[line width=1.5pt,->]
				%	(1) edge [bend right=20] node {} (4)
				%	(4) edge [bend right=20] node {} (1)
				%	(2) edge [bend right=20] node {} (4)
				%	(4) edge [bend right=20] node {} (2);
				\draw[->,rounded corners=15pt] (1) -- (1.2, 0.4) -- (4);
				\draw[->,rounded corners=15pt] (4) -- (0.8, 0.8) -- (1);
				\draw[->,rounded corners=15pt] (4) -- (2.8, 0.4) -- (2);
				\draw[->,rounded corners=15pt] (2) -- (3.2, 0.8) -- (4);
				\draw[->,rounded corners=15pt] (3) -- (1.7, 2.3) -- (4);
				\draw[->,rounded corners=15pt] (4) -- (2.2, 2.3) -- (3);
				\draw[->,rounded corners=20pt] (3) -- (3.4, 1.8) -- (2);
				\draw[->,rounded corners=20pt] (2) -- (2, -0.4) -- (1);
				\draw[->,rounded corners=20pt] (1) -- (0.6, 1.8) -- (3);
				%\node at (1.4,0.3) {$e_1$};
				%\node at (0.85, 1.15) {$e_2$};
				%\node at (1.5,1.9) {$e_3$};
				%\node at (2.4,1.9) {$e_4$};
				%\node at (2.8,1.1) {$e_5$};
				%\node at (2.7,0.3) {$e_6$};
				%\node at (2,-0.7) {$e_7$};
				%\node at (0.6, 2.5) {$e_8$};
				%\node at (3.4, 2.5) {$e_9$};
				%\node at (-0.2, 3.5) {$\cE$};
				\draw[->] (-2, 1.5) -- (-1, 1.5);
			\end{scope}
		\end{tikzpicture}
	\end{center}
	\caption{An example for the tour-rotor action. See Example \ref{ex:tour_rotor_action} for the details. The ribbon structure is given by the plane embedding. \label{f:tour_rotor_action}}
\end{figure}
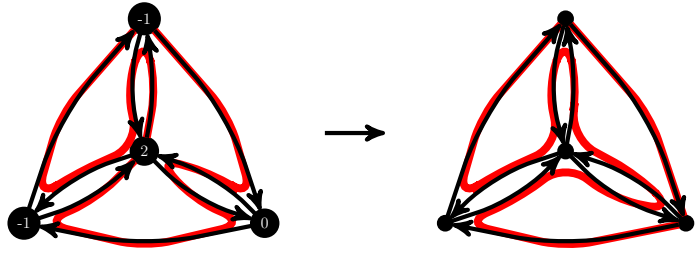

\begin{proof}[Proof of Theorem \ref{thm:Eulerian_tour_action_well_defined}]
	By Proposition \ref{prop:alt_def_rotor_action}, 
	$r_u(\mathbf{x}, A_{\overrightarrow{uv}_e}(\cE))=A_{\overrightarrow{uv}_e}(\cE')$ implies $$(\mathbf{x}, A_{\overrightarrow{uv}_e}(\mathcal{E})\cup \overrightarrow{uv}_e)\sim (\mathbf{0}, A_{\overrightarrow{uv}_e}(\cE')\cup \overrightarrow{uv}_e).$$
	
	Lemma \ref{lem:change_of_arborescences_when_changing_first_edge} implies $(\mathbf{1}_v-\mathbf{1}_z, A_{\overrightarrow{uv}_e}(\mathcal{E})\cup \overrightarrow{uv}_e)\sim (\mathbf{0}, A_{\overrightarrow{wz}_f}(\cE)\cup \overrightarrow{wz}_f).$
	
	Similarly, $(\mathbf{1}_v-\mathbf{1}_z, A_{\overrightarrow{uv}_e}(\mathcal{E}')\cup \overrightarrow{uv}_e)\sim (\mathbf{0}, A_{\overrightarrow{wz}_f}(\cE')\cup \overrightarrow{wz}_f).$
	
	Hence, 
	\begin{align*}
		(\mathbf{x}, A_{\overrightarrow{wz}_f}(\mathcal{E})\cup \overrightarrow{wz}_f) \sim (\mathbf{x}+\mathbf{1}_v-\mathbf{1}_z, A_{\overrightarrow{uv}_e}(\mathcal{E})\cup \overrightarrow{uv}_e) \\  \sim (\mathbf{1}_v-\mathbf{1}_z, A_{\overrightarrow{uv}_e}(\mathcal{E}')\cup \overrightarrow{uv}_e)  \sim (\mathbf{0}, A_{\overrightarrow{wz}_f}(\cE')\cup \overrightarrow{wz}_f).
	\end{align*}
	
	By Proposition \ref{prop:alt_def_rotor_action}, this implies $r_w(\mathbf{x}, A_{\overrightarrow{wz}_f}(\cE))=A_{\overrightarrow{wz}_f}(\cE')$.
\end{proof}

\section{Special case: Rotor routing in undirected ribbon graphs}
\label{s:action_of_bidirected_graphs}

Let us see what the results of Section \ref{sec:tour-rotor_action} imply for the rotor-routing action of undirected ribbon graphs.

Let $G$ be an undirected ribbon graph, and let $G^{\leftrightarrows}$ be the associated normal bidirected ribbon graph. Recall that $G$ is planar if and only if $G^{\leftrightarrows}$ is.

By Corollary \ref{cor:Euler_sandpile_acts_on_tours}, the tour-rotor action of $S(G^{\leftrightarrows})$ is a canonical action of $S(G)$ on the compatible Eulerian tours of $G^{\leftrightarrows}$, which is a set of objects equinumerous with spanning trees.
Moreover, via the bijection $\cE \mapsto A_{\overrightarrow{vw}_e}(\cE)$, this action descends to the rotor-routing action with root $v$.

\subsection{Basepoint independence revisited} \label{ssec:bidirected_case}

The results of Section \ref{sec:tour-rotor_action} enable us to phrase an alternative condition for the basepoint-independence of the rotor-routing action using the tour-rotor action. 
Recall that for a spanning tree $T$, we denote the corresponding in-arborescence of $G^{\leftrightarrows}$ rooted at $v$ by $T^v$.

\begin{prop}\label{prop:root_independence_consition_using_tour_action}
	Let $G$ be a ribbon graph, and $u,w\in V(G)$. The followings are equivalent:
	\begin{itemize}
		\item[$(i)$] The rotor-routing actions of $G$ (Definition \ref{def:rr_action_undirected}) with root $u$ and with root $w$ agree.
		\item[$(ii)$] For each pair of arcs $\overrightarrow{uv}_e, \overrightarrow{wz}_f \in E(G^{\leftrightarrows})$, there exists some $\mathbf{y}\in S(G)$ such that for each spanning tree $T$ of $G$, we have $$r(\mathbf{y},A^{-1}_{\overrightarrow{uv}_e}(T^u))=A^{-1}_{\overrightarrow{wz}_f}(T^w).$$
		\item[$(iii)$] There exist arcs $\overrightarrow{uv}_e, \overrightarrow{wz}_f \in E(G^{\leftrightarrows})$, and $\mathbf{y}\in S(G)$ such that for each spanning tree $T$ of $G$, we have %$$r(y,\cE_T^{\overrightarrow{uv}})=\cE_T^{\overrightarrow{wz}}.$$
		$$r(\mathbf{y},A^{-1}_{\overrightarrow{uv}_e}(T^u))=A^{-1}_{\overrightarrow{wz}_f}(T^w).$$
	\end{itemize}
\end{prop}

\begin{proof}
	$(i) \Rightarrow (ii)$: Take arbitrary edges $\overrightarrow{uv}_e$ and $\overrightarrow{wz}_f$. Let $T_1$ be an arbitrary fixed spanning tree, and 
	take $\mathbf{y}\in S(G)$ such that $$r(\mathbf{y},A^{-1}_{\overrightarrow{uv}_e}(T_1^u))=A^{-1}_{\overrightarrow{wz}_f}(T_1^w).$$ 
	Such a $\mathbf{y}$ uniquely exists because the tour-rotor action is simply transitive.
	
	Let $T_2$ be another arbitrary spanning tree of $G$, and suppose that $r_u(\mathbf{x}, T_1^u)=T_2^u$.  
	The rotor-routing actions with roots $u$ and $w$ agree if and only if $r_w(\mathbf{x}, T_1^w)=T_2^w$. 
	
	By definition, $r(\mathbf{y},A^{-1}_{\overrightarrow{uv}_e}(T_1^u))=A^{-1}_{\overrightarrow{wz}_f}(T_1^w)$ means
	$$r_u(\mathbf{y},A_{\overrightarrow{uv}_e}(A^{-1}_{\overrightarrow{uv}_e}(T_1^u)))=A_{\overrightarrow{uv}_e}(A^{-1}_{\overrightarrow{wz}_f}(T_1^w)),$$ or in a more useful form, 
	$$(\mathbf{y},T_1^u\cup \overrightarrow{uv}_e)\sim (\mathbf{0}, A_{\overrightarrow{uv}_e}(A^{-1}_{\overrightarrow{wz}_f}(T_1^w))\cup \overrightarrow{uv}_e).$$ 
	This implies 
	$$(\mathbf{x}+\mathbf{y}+\mathbf{1}_v-\mathbf{1}_z,T_1^u\cup \overrightarrow{uv}_e)\sim (\mathbf{x}+\mathbf{1}_v-\mathbf{1}_z, A_{\overrightarrow{uv}_e}(A^{-1}_{\overrightarrow{wz}_f}(T_1^w))\cup \overrightarrow{uv}_e).$$ 
	The assumption $r_u(\mathbf{x}, T_1^u)=T_2^u$ implies $$(\mathbf{x}+\mathbf{y}+ \mathbf{1}_v-\mathbf{1}_z, T_1^u \cup \overrightarrow{uv}_e)\sim (\mathbf{y}+ \mathbf{1}_v-\mathbf{1}_z, T_2^u \cup \overrightarrow{uv}_e).$$
	Furthermore, by Lemma \ref{lem:change_of_arborescences_when_changing_first_edge}, we have $$(\mathbf{y}+ \mathbf{1}_v-\mathbf{1}_z, T_2^u \cup \overrightarrow{uv}_e) = (\mathbf{y}+ \mathbf{1}_v-\mathbf{1}_z, A_{\overrightarrow{uv}_e}(A^{-1}_{\overrightarrow{uv}_e}(T_2^u)) \cup \overrightarrow{uv}_e) \sim (\mathbf{y}, A_{\overrightarrow{wz}_f}(A^{-1}_{\overrightarrow{uv}_e}(T_2^u)) \cup \overrightarrow{wz}_f).$$
	Using Lemma \ref{lem:change_of_arborescences_when_changing_first_edge}, we further have $$(\mathbf{x}+\mathbf{1}_v-\mathbf{1}_z, A_{\overrightarrow{uv}_e}(A^{-1}_{\overrightarrow{wz}_f}(T_1^w))\cup \overrightarrow{uv}_e)\sim (\mathbf{x}, A_{\overrightarrow{wz}_f}(A^{-1}_{\overrightarrow{wz}_f}(T_1^w))\cup \overrightarrow{wz}_f) = (\mathbf{x}, T_1^w\cup \overrightarrow{wz}_f).$$ 
	Finally, $r_w(\mathbf{x},T^w_1)=T_2^w$ implies \begin{equation}\label{eq:root_at_w}
		(\mathbf{x}, T_1^w\cup \overrightarrow{wz}_f) \sim (\mathbf{0}, T_2^w\cup \overrightarrow{wz}_f).
	\end{equation} %if and only if $r_w(x,T^w_1)=T_2^w$.
	
	Putting the parts together, we have $$(\mathbf{y}, A_{\overrightarrow{wz}_f}(A^{-1}_{\overrightarrow{uv}_e}(T_2^u)) \cup \overrightarrow{wz}_f)
	\sim (\mathbf{0}, T_2^w \cup \overrightarrow{wz}_f) = (\mathbf{0}, A_{\overrightarrow{wz}_f}(A^{-1}_{\overrightarrow{wz}_f}(T_2^w)))\cup \overrightarrow{wz}_f),$$ %if and only if $r_w(x,T^w_1)=T_2^w$.
	hence $r(\mathbf{y},A^{-1}_{\overrightarrow{uv}_e}(T_2^u))=A^{-1}_{\overrightarrow{wz}_f}(T_2^w)$.
	
	$(ii)\Rightarrow (iii)$: Immediate.
	
	$(iii) \Rightarrow (i)$: Note that \eqref{eq:root_at_w} also implies $r_w(\mathbf{x},T_1^w)=T_2^w$. Hence if $(iii)$ is satisfied, we can do essentially the same computation as in the $(i)\Rightarrow (ii)$ case and obtain that the rotor-routing actions with root $u$ and $w$ agree.
\end{proof}

By the the result of Chan, Church and Grochow \cite[Theorem 2]{Chan15}, the condition of Proposition \ref{prop:root_independence_consition_using_tour_action} is satisfied if and only if $G$ is embedded into the plane.
We give yet another proof for \cite[Theorem 2]{Chan15}. Even though the main ideas behind this new proof are basically the same as in \cite{Chan15}, we feel that technically this proof is easier, and it gives some structural insight (in particular, on the values of $y$) which might be interesting.

As a lemma, we use a statement similar to \cite[Proposition 9]{Chan15}. A rotor configuration $\varrho$ is called a \emph{unicycle} if there is a unique cycle in $\{\varrho(v) : v\in V\}$.

\begin{lemma}\label{lem:unicycle_reversible}
	Let $G$ be a ribbon graph without loops. Let $\varrho$ be a unicycle of $G^{\leftrightarrows}$, and let $\overleftarrow{\varrho}$ be the unicycle obtained by reversing the cycle in $\varrho$. Then, $(\mathbf{0},\varrho) \sim (\mathbf{0}, \overleftarrow{\varrho})$ if and only if the unique cycle in $\varrho$ is separating.
\end{lemma}
The statement is not true if we have loops, but as the addition or removal of loops do not change the sandpile group, it is not a big restriction to suppose that we have no loops.
\begin{proof}
	Let $C\subset G^{\leftrightarrows}$ be the cycle of $\varrho$. First, suppose that $C$ is separating, that is, it has inside and outside.
	By symmetry, we can assume that the opposite 
	versions of the edges of $C$ are in the inside of $C$, and if we do a routing at a vertex of $C$, the rotor turns to the inside of $C$. %and eventually into  $\overleftarrow{\varrho}(v)$. 
	For each $v\in C$, do as many routings such that $\varrho(v)$ turns into $\overleftarrow{\varrho}(v)$. Also, for vertices $u$ inside $C$, do $d^+(u)$ routings (that is, one full turn). Then, the rotors moved from $\varrho$ to $\overleftarrow{\varrho}$. We need to check that the chip configuration remains $\mathbf{0}$. The chip movements are the following: a chip passes along each edge inside $C$ (not including edges of $C$). Note that the subgraph of edges inside $C$ is Eulerian. (It is the normal bidirected graph of an undirected graph, minus the consistently oriented cycle $C$.) Hence the effect of the chip movements indeed cancels out. 
	
	Now let us show that $(\mathbf{0},\varrho) \sim (\mathbf{0}, \overleftarrow{\varrho})$ implies that $C$ is separating. For each vertex $v\in C$, let us call an out-edge between $\varrho(v)$ and $\overleftarrow{\varrho}(v)$ (in the cyclic order of out-edges of $v$) a \emph{left edge}, and an edge between $\overleftarrow{\varrho}(v)$ and $\varrho(v)$ a \emph{right edge}. (Slightly abusing notation, we refer to the underlying undirected edge also as a left or right edge.) Let $L\subset E(G^{\leftrightarrows})$ denote the set of all left edges of vertices of $C$. Then, for $\mathbf{y}=\sum_{e\in L} \chi_e$, we have $(\mathbf{0},\varrho)\sim (\mathbf{y},\overleftarrow{\varrho})$. Indeed, for $\mathbf{y}'=\sum_{e\in L} \chi_e + \sum_{v\in V(C)} \chi_{\overleftarrow{\varrho}(v)}$, if we start from $(\mathbf{0},\varrho)$, and route each vertex $v$ until the rotor turns into $\overleftarrow{\varrho}(v)$, then we send a chip along each edge in $L$, and also along $\overleftarrow{\varrho}(v)$, hence the chip-and-rotor configuration becomes $(\mathbf{y}',\overleftarrow{\varrho})$. Note also that $\sum_{v\in V(C)} \chi_{\overleftarrow{\varrho}(v)}=0$, hence $\mathbf{y}'=\mathbf{y}$.
	
	To show that $C$ is separating, it is enough to find a partition of $V-V(C)$ into two sets $S$ and $T$ such that no edge connects $S$ and $T$, and each edge connecting $S$ to $V(C)$ is a left edge, while each edge connecting $T$ to $V(C)$ is a right edge.
	By Fact \ref{fact:chip_lin_ekv_and_rotor_lin_ekv}, we have $\mathbf{y}\sim \mathbf{0}$. By Proposition \ref{prop:zero_lin_ekv_cond}, this means that $L$ can be written as a disjoint union of elementary directed cuts and directed cycles. Let the directed cuts be $J_1, \dots J_k$, and the cycles be $C_1, \dots C_\ell$. Note that since all edges of $L$ have their tail on $C$, if $e\in C_i$, then both endpoints of $e$ are on $C$. Remember that in this setting, by elementary directed cut $J_i$ we mean that there is a partition $U_i \sqcup W_i$ of the vertex set $V(G^{\leftrightarrows})$ such that $U_i$ and $W_i$ both span connected subgraphs, and the edges of $J_i$ are the edges leading from $U_i$ to $W_i$. (There are also edges leading from $W_i$ to $U_i$, but these are not considered part of the edge set $J_i$.) 
	
	We claim that $V(C)\subseteq U_i$ for each $i$. Indeed, $V(C)\cap U_i\neq\emptyset$ since the tails of edges of $J_i$ are in $V(C)\cap U_i$. If we had $V(C) \not \subset U_i$, then some edges of $C$ were also part of $J_i$, but $J_i\subset L$ so this is not possible. 
	
	Take the vertex sets $S:=W_1\cup \dots \cup W_k$, and $T:=V-S-V(C)$.
	We claim that no edge connects a vertex of $S$ to a vertex of $T$. Indeed, if there were an edge $e$ connecting a vertex $u\in S$ and $v\in T$, then $u\in W_i$ for some $i$, and $v\notin W_i$, hence $v\in U_i$. But in this case, $\overrightarrow{vu}_e\in J_i$, and since $J_i \subseteq L$, this means that $v\in V(C)$, a contradiction with $v\in T$.
	
	Next, we show that an edge $e$ connecting $S$ to $V(C)$ is a left edge. Let $u\in S$ and $v\in V(C)$ be adjacent vertices. Again, $u\in W_i$ for some $i$. As $V(C)\subseteq U_i$, we have $v\in U_i$, and hence $\overrightarrow{vu}_e\in J_i\subseteq L$, as claimed.
	
	Finally, we show that an edge $e$ connecting $T$ to $V(C)$ is a right edge. Let $u\in T$ and $v\in V(C)$. As $u\notin V(C)$, the edge $e$ is either a left edge or a right edge.  If $\overrightarrow{vu}_e$ were a left edge, then $\overrightarrow{vu}_e\in J_i$ for some $i$. This is true because $u\notin V(C)$, and hence we cannot have $e\in C_j$. But this would mean $u\in W_i$, contradicting the assumption  $u\in T$.
	
	These imply that in the cellular embedding corresponding to the ribbon structure, $C$ is a separating cycle, separating the vertex sets $S$ and $T$.
\end{proof}

\begin{thm}\cite[Theorem 2]{Chan15}
	Let $G$ be a ribbon graph without loops. The rotor routing action (Definition \ref{def:rr_action_undirected}) of $G$ is independent of the root if and only if $G$ is planar.    
\end{thm}
\begin{proof} 
	We start by showing that the action is root-independent for plane graphs.
	As $G$ is connected, it is enough to show that for adjacent vertices $u$ and $v$, the rotor-routing actions with basepoints $u$ and $v$ agree. 
	
	In this proof we use the notation that if $e$ is an edge of $G$ (the undirected graph), then $\overrightarrow{uv}_e$ and $\overrightarrow{vu}_e$ are the two corresponding edges of $G^{\leftrightarrows}$.
	
	We apply Proposition \ref{prop:root_independence_consition_using_tour_action} $(iii) \Rightarrow (i)$ with $\overrightarrow{uv}_e$, and $\overrightarrow{vz}_f=\nex(v,\overrightarrow{uv}_e)$.
	
	Take an arbitrary spanning tree $T$. 
	
	$r(\mathbf{y}, A^{-1}_{\overrightarrow{uv}_e}(T^u))=A^{-1}_{\overrightarrow{vz}_f}(T^v)$ is equivalent to 
	$$(\mathbf{y}, T^u \cup \overrightarrow{uv}_e)=(\mathbf{y}, A_{\overrightarrow{uv}_e}(A^{-1}_{\overrightarrow{uv}_e}(T^u))\cup \overrightarrow{uv}_e) \sim (\mathbf{0}, A_{\overrightarrow{uv}_e}(A^{-1}_{\overrightarrow{vz}_f}(T^v))\cup \overrightarrow{uv}_e).
	$$
	By Lemma \ref{lem:change_of_arborescences_when_changing_first_edge} we have $$(\mathbf{0}, A_{\overrightarrow{uv}_e}(A^{-1}_{\overrightarrow{vz}_f}(T^v))\cup \overrightarrow{uv}_e) \sim (\mathbf{1}_z-\mathbf{1}_v, A_{\overrightarrow{vz}_f}(A^{-1}_{\overrightarrow{vz}_f}(T^v))\cup \overrightarrow{vz}_f)=(\mathbf{1}_z-\mathbf{1}_v, T^v\cup \overrightarrow{vz}_f).$$
	
	Route $d_G(v)-1$ times the vertex $v$ in $(\mathbf{1}_z-\mathbf{1}_v, T^v\cup \overrightarrow{vz}_f)$. This transforms the rotor configuration to $T^v\cup \overrightarrow{vu}_e$, where $\overrightarrow{vu}_e$ is the ``opposite edge'' of $\overrightarrow{uv}_e$ (since we choose $\overrightarrow{vz}_f=\nex(v, \overrightarrow{uv}_e)$), and the chip configuration becomes $\sum\{ \chi_{\overrightarrow{vw}}: \overrightarrow{vw} \text{ is an out-edge of } v\} \sim \mathbf{0}$.
	
	Hence we need to find $\mathbf{y}$ such that $$(\mathbf{y}, T^u \cup \overrightarrow{uv}_e)\sim (\mathbf{0}, T^v\cup \overrightarrow{vu}_e).$$
	Notice that $\varrho=T^u \cup \overrightarrow{uv}_e$ is a unicycle, and $T^v \cup \overrightarrow{vu}_e$ is its reversal $\overleftarrow{\varrho}$. Hence by Lemma \ref{lem:unicycle_reversible}, we have $(\mathbf{0}, T^u \cup \overrightarrow{uv}_e) \sim (\mathbf{0}, T^v \cup \overrightarrow{vu}_e)$, thus, $\mathbf{y}=\mathbf{0}$, regardless of the choice of $T$.
	
	Next, we show that for non-planar $G$, the rotor-routing action is not independent of the root. Notice that the above proof showing that 
	$r(\mathbf{0}, A^{-1}_{\overrightarrow{uv}_e}(T^u))=A^{-1}_{\overrightarrow{vz}_f}(T^v)$ works for any ribbon graph if the unique cycle in $\varrho=T^u\cup \overrightarrow{uv}_e$ is separating. Also, Lemma \ref{lem:unicycle_reversible} shows that if the unique cycle in $\varrho$ is not separating, then we cannot have $r(\mathbf{0}, A^{-1}_{\overrightarrow{uv}_e}(T^u))=A^{-1}_{\overrightarrow{vz}_f}(T^v)$. Hence we are ready if we can show two trees $T_1$ and $T_2$ such that the cycle in $T_1^u\cup \overrightarrow{uv}_e$ is a separating cycle, while the cycle in $T_2^u\cup \overrightarrow{uv}_e$ is nonseparating. If $G$ is nonplanar, then it has a nonseparating cycle $C$. Choose one edge $e=uv\in C$ and extend $C-e$ to a spanning tree $T_2$. For $T_1$, choose any tree such that $e\in T_1$ (which can be done since $e$ is not a loop). Then the unique cycle in $T_1^u\cup \overrightarrow{uv}_e$ is $\{\overrightarrow{uv}_e, \overrightarrow{vu}_e\}$, which is separating.
\end{proof}

It would also be interesting to see what kinds of sandpile elements $\mathbf{y}\in S(G)$ appear for nonplanar graphs.

\begin{prob}\label{prob:y_homologies}
	Do the elements $\mathbf{y}\in S(G)$ appearing as $r(\mathbf{y}, A^{-1}_{\overrightarrow{uv}_e}(T^u))=A^{-1}_{\overrightarrow{vz}_f}(T^v)$ satisfy nice algebraic conditions? Are they connected to the homologies of $\Sigma$?
\end{prob}

\subsection{An alternative formalism for the rotor routing action}

Another consequence of the results in Section \ref{sec:tour-rotor_action} is that we obtain an alternative definition for the rotor-routing action, that is more similar to the definition of the Bernardi action than the original one.

We recall the definition of the Bernardi action of the sandpile group $S(G)$ of an undirected graph $G$. For us, only the big picture is important here. For the exact definitions, see \cite{Baker-Wang}.
The Bernardi action is defined the following way: %The break divisors form a 
$S(G)$ acts (canonically) on the set of break divisors of $G$ by addition. The Bernardi bijection $\beta_{v,e}$ is a bijection between break divisors and spanning trees, that depends on a ribbon structure and a fixed vertex $v$ and an incident edge $e=vu$. The Bernardi action is the composition of the canonical action on break divisors with the Bernardi bijection. It is known that for a fixed ribbon structure, the action does not depend on the fixed edge $e$, only on the vertex $v$.

Originally, the definition of the rotor-routing action is very different to the definition of the Bernardi action. However, given the tour-rotor action, we can give a very similar definition.

\begin{defn}[Alternative definition of the rotor-routing action]
	$$r_v(x,T) = A_{\overrightarrow{uv}_e}(r(x,A^{-1}_{\overrightarrow{uv}_e}(T^u))).$$
\end{defn}
That is, we can also write the rotor routing action as the composition of a canonical action on some set of objects, and a bijection between those objects and the spanning trees, depending on a fixed vertex and edge. Here, too, the fixed edge does not make a difference.

Of course this alternative definition is in some sense ``cheating'', because we used $r_v$ to define the tour-rotor action. But it still enables us to see additional structure behind changing the root vertex. This gives hope to uncover a finer relationship between Bernardi and rotor-routing actions in the non-planar case.

\begin{prob}\label{prob:finer_relationships}
	Use the alternative definition of the rotor-routing action to uncover a finer relationship between the rotor-routing and Bernardi actions on non-planar graphs.
\end{prob}

We note another interesting fact pointed out by Tamás Kálmán: For the Bernardi action, the set of break divisors does not depend on the ribbon structure, while the Bernardi bijections do depend on it. On the other hand, for the rotor-routing action, the set of compatible Eulerian tours depends on the ribbon structure, while the bijections between compatible Eulerian tours and spanning trees do not depend on the ribbon structure.

\section{Special case: The sandpile group of the medial digraph acts canonically on quasi-trees}
\label{s:medial}

For a graph embedded into a closed orientable surface $\Sigma$, there is a natural analogue to spanning trees: spanning quasi-trees. A spanning quasi-tree is a spanning subgraph, whose $\varepsilon$-neighborhood (within $\Sigma$) has one boundary component.

If the graph is embedded into the plane (sphere), then the quasi-trees are exactly the spanning trees. However, if $\Sigma$ is not the sphere, the two notions diverge. Spanning trees are always quasi-trees, but if $\Sigma$ is not the sphere, new quasi-trees appear. See Figure \ref{f:embedded_graph}, where the thick red edges on the right panel show a spanning quasi-tree that is not a spanning tree.

Spanning trees are the bases of the graphic matroid. Analogously to this, embedded graphs have a delta-matroid structure, and quasi-trees are the bases of this delta-matroid. %of the embedded graph.

Recently, Merino, Moffatt and Noble defined a Jacobian group for embedded graphs \cite{merino2023critical}, whose cardinality is the number of quasi-trees. Baker, Ding and Kim \cite{BDK} showed that this group acts canonically on the quasi-trees. (We note that \cite{BDK} also generalizes the Jacobian to regular orthogonal matroids, but in this paper we only consider embedded graphs.)

Here, we point out that the tour-rotor action from Section \ref{sec:tour-rotor_action} (if applied appropriately) also gives a canonical free, transitive action on the quasi-trees. More precisely, the compatible Eulerian tours of the so-called medial digraph of $G$ are in canonical bijection with the quasi-trees of $G$ (due to Bouchet \cite{Bouchet}). Hence by Corollary \ref{cor:Euler_sandpile_acts_on_tours}, the sandpile group of the medial digraph has a canonical free, transitive action on the quasi-trees of $G$. 

We will show that in fact the sandpile group of the medial digraph is canonically isomorphic to the Jacobian of the embedded graph, and moreover, the tour-rotor action of the sandpile group of the medial digraph also agrees with the Bernardi action of the Jacobian by Baker, Ding and Kim \cite{BDK}.

We note that Baker, Ding and Kim need a reference orientation to define the Jacobian, while no reference orientation is needed for defining the sandpile group of the medial digraph.

After this summary, let us give the exact definitions. 

\subsection{Embedded graphs, medial digraphs, and the action on quasi-trees}

Let $G$ be a ribbon graph. We will think of it as cellularly embedded into the closed orientable surface $\Sigma$.
A dual $G^*$ of $G$ can be defined with respect to $\Sigma$: Put a vertex of $G^*$ inside each country of the embedding (in other words, in each component of $\Sigma\setminus G$), and connect dual vertices of neighboring countries through each edge of $G$. This way, the edge set $E(G^*)$ of $G^*$ is in bijection with the edge set $E(G)$ of $G$: for an edge $e\in E(G)$, the corresponding edge $e^*\in E(G^*)$ is the one that intersects $e$.

Each pair of dual edges $e\in E(G)$ and $e^*\in E(G^*)$ intersect in a point. Let us call this point $v_e$. To avoid confusing these nodes with vertices of $G$ or $G^*$, we will call them \emph{emerald nodes}. Here, ``emerald'' stands for ``edge''. These nodes will be emerald on the figures. Also, vertices of $G$ will be violet (for vertex), and vertices of $G^*$ will be red (for region). See Figure \ref{f:embedded_graph}. 

\begin{figure}
	\begin{center}
		\begin{tikzpicture}[-,>=stealth',auto,scale=0.75, thick]
			\tikzstyle{r}=[circle,scale=0.5,fill,draw,color=red]
			\tikzstyle{b}=[circle,scale=0.5,fill,draw,color=v]
			\tikzstyle{e}=[circle,scale=0.5,fill,draw,color=green]
			\begin{scope}[shift={(-4,0)}]
				\draw  (0, 0) rectangle (5,4);                \draw[->] (0,1) -- (0,1.5);
				\draw[->] (5,1) -- (5,1.5);
				\draw[->>] (1.5, 0) -- (2.8, 0);
				\draw[->>] (1.5, 4) -- (2.8, 4);
				\node[b] (2) at (2, 1) {};
				\node[b] (1) at (4, 1) {};
				\node[r] (3) at (1, 3) {};
				\node[r] (4) at (3, 3) {};
				\draw[-,color=red] (2) -- (1);
				\draw[-,color=red] (1) -- (5, 1);
				\draw[-,color=red] (0, 1) -- (2);
				\draw[-,color=red] (1) -- (4, 4);
				\draw[-,color=red] (2) -- (2, 4);
				\draw[-,color=red] (4, 0) -- (1);
				\draw[-,color=red] (2, 0) -- (2);
				\draw[-,color=v] (4) -- (3);
				\draw[-,color=v] (3) -- (0, 3);
				\draw[-,color=v] (5, 3) -- (4);
				\draw[-,color=v] (3, 0) -- (4);
				\draw[-,color=v] (1, 0) -- (3);
				\draw[-,color=v] (4) -- (3, 4);
				\draw[-,color=v] (3) -- (1, 4);
				\node[e] (5) at (1, 1) {};
				\node[e] (6) at (2, 3) {};
				\node[e] (7) at (3, 1) {};
				\node[e] (8) at (4, 3) {};
				%median elek
				\draw[->] (6) -- (5);
				\draw[-] (5) -- (0,2);
				\draw[->] (5,2) -- (8);
				\draw[->] (6) -- (5);
				\draw[->] (7) -- (6);
				\draw[->] (8) -- (7);
				\draw[->] (3.5, 4) -- (8);
				\draw[-] (7) -- (3.5, 0);
				\draw[-] (8) -- (5, 3.5);
				\draw[-] (0,3.5) -- (0.5, 4);
				\draw[->] (0.5, 0) -- (5);
				\draw[-] (5) -- (1.5, 0);
				\draw[->] (1.5, 4) -- (6);
				\draw[-] (6) -- (2.5,4);
				\draw[->] (2.5,0) -- (7);
				%\draw[color=gray] (0, 3.5) -- (0.5, 4);
				%\draw[color=gray,rounded corners=15pt] (5, 2) -- (4, 3) -- (3, 1) -- (2, 3) -- (1, 1) -- (1.5, 0);
				%\draw[color=gray,rounded corners=15pt] (0.5, 0) -- (1, 1) -- (0, 2);
				%\draw[color=gray,rounded corners=15pt] (2.5, 0) -- (3, 1) -- (3.5, 0);
				%\draw[color=gray,rounded corners=15pt] (1.5, 4) -- (2, 3) --       (2.5, 4);
				%\draw[color=gray,rounded corners=12pt] (3.5, 4) -- (4, 3) -- (5, 3.5);
			\end{scope}
			\begin{scope}[shift={(4,0)}]
				\draw  (0, 0) rectangle (5,4);
				\draw[->] (0,1) -- (0,1.5);
				\draw[->] (5,1) -- (5,1.5);
				\draw[->>] (1.5, 0) -- (2.8, 0);
				\draw[->>] (1.5, 4) -- (2.8, 4);
				\node[b] (2) at (2, 1) {};
				\node[b] (1) at (4, 1) {};
				\node[r] (3) at (1, 3) {};
				\node[r] (4) at (3, 3) {};
				\draw[-,color=red,line width=3.5pt] (2) -- (1);
				\draw[-,color=red,line width=3.5pt] (1) -- (5, 1);
				\draw[-,color=red,line width=3.5pt] (0, 1) -- (2);
				\draw[-,color=red] (1) -- (4, 4);
				\draw[-,color=red,line width=3.5pt] (2) -- (2, 4);
				\draw[-,color=red] (4, 0) -- (1);
				\draw[-,color=red,line width=3.5pt] (2, 0) -- (2);
				\draw[-,color=v] (4) -- (3);
				\draw[-,color=v] (3) -- (0, 3);
				\draw[-,color=v] (5, 3) -- (4);
				\draw[-,color=v] (3, 0) -- (4);
				\draw[-,color=v] (1, 0) -- (3);
				\draw[-,color=v] (4) -- (3, 4);
				\draw[-,color=v] (3) -- (1, 4);
				\node[e] (5) at (1, 1) {};
				\node[e] (6) at (2, 3) {};
				\node[e] (7) at (3, 1) {};
				\node[e] (8) at (4, 3) {};
				\draw[color=gray,line width=3.5pt] (0, 3.5) -- (0.5, 4);
				\draw[color=gray,rounded corners=18pt,line width=3.5pt] (1.5, 4) -- (2, 3) -- (1, 1) -- (0, 2);
				\draw[color=gray,rounded corners=18pt,line width=3.5pt] (0.5, 0) -- (1, 1) -- (1.5, 0);
				\draw[color=gray,rounded corners=18pt,line width=3.5pt] (2.5, 0) -- (3, 1) -- (3.5, 0);
				\draw[color=gray,rounded corners=18pt,line width=3.8pt] 
				(5, 2) -- (4, 3) -- (3, 1) -- (2, 3) -- (2.5, 4);
				\draw[color=gray,rounded corners=18pt,line width=3.8pt] (3.5, 4) -- (4, 3) -- (5, 3.5);
				%median elek
				\draw[->,dashed] (6) -- (5);
				\draw[-] (5) -- (0,2);
				\draw[->] (5,2) -- (8);
				\draw[->] (6) -- (5);
				\draw[->] (7) -- (6);
				\draw[->] (8) -- (7);
				\draw[->] (3.5, 4) -- (8);
				\draw[-] (7) -- (3.5, 0);
				\draw[-] (8) -- (5, 3.5);
				\draw[-] (0,3.5) -- (0.5, 4);
				\draw[->] (0.5, 0) -- (5);
				\draw[-] (5) -- (1.5, 0);
				\draw[->] (1.5, 4) -- (6);
				\draw[-] (6) -- (2.5,4);
				\draw[->] (2.5,0) -- (7);
			\end{scope}
		\end{tikzpicture}
	\end{center}
	\caption{Left panel: A graph $G$ embedded into the torus (violet vertices, red edges), its dual $G^*$ (red vertices, violet edges), and the medial digraph $G^{\bowtie}$ (emerald nodes, black edges). Right panel: A quasi-tree of $G$ (thick red edges) and the corresponding Eulerian tour of $G^{\bowtie}$ (indicated by a grey curve).	\label{f:embedded_graph}}
\end{figure}
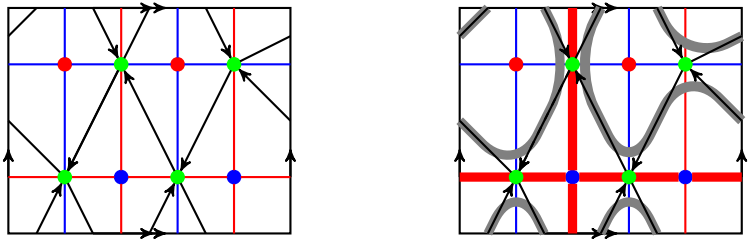

The medial digraph of the ribbon graph $G$ is a balanced Eulerian ribbon digraph, denoted by $G^{\bowtie}$. The node set of $G^{\bowtie}$ is $\{v_e \mid e\in E(G)\}$.
An edge of $G^{\bowtie}$ leads from $v_e$ to $v_f$, if the edges $e$ and $f$ are incident to a common vertex $v$, and at that vertex, $f$ is the edge following $e$ in the ribbon structure.
This way, each emerald node $v_e$ has two in-edges and two out-edges, embedded into $\Sigma$ in an alternating way. Around $v_e$, in positive cyclic order, we see a half-edge of $e$, an in-edge of $v_e$, a half-edge of $e^*$, an out-edge of $v_e$, a half-edge of $e$, an in-edge of $v_e$, a half-edge of $e^*$, and finally, an out-edge of $v_e$.
See Figure \ref{f:embedded_graph} for an example. 

By Bouchet \cite[Corollary 3.4]{Bouchet}, the Eulerian tours of $G^{\bowtie}$ are in canonical bijection with the quasi-trees of $G$. 
\begin{bij}\label{bij_tour_quasi-tree}
	To an Eulerian tour of $G^{\bowtie}$, one can associate the following quasi-tree of $G$: 
	Take any emerald node $v_e$. As the tour arrives to $v_e$ through one of the in-edges of $v_e$, it either continues with the out-edge that is on the same side of $e$ and on the opposite side of $e^*$, or with the out-edge that is on the opposite side of $e$, and on the same side of $e^*$ as the in-edge. That is, the Euler tour either crosses $e$ and does not cross $e^*$ or vice versa. Associate to the tour the edge set $Q\subseteq E(G)$ consisting of the edges $e$ that are not crossed by the tour (at $v_e$).
\end{bij} 

By Bouchet \cite{Bouchet}, this is a bijection. The Eulerian tour is homotopic to the $\varepsilon$-boundary of the quasi-tree in $\Sigma$. See the right panel of Figure \ref{f:embedded_graph} for an example.

Note that for $G^{\bowtie}$, each emerald node has out-degree two, hence each Eulerian tour of $G^{\bowtie}$ is automatically compatible with the embedding. Corollary \ref{cor:Euler_sandpile_acts_on_tours} implies:
\begin{thm}
	For an embedded graph $G$, via the tour-rotor action and Bijection \ref{bij_tour_quasi-tree}, $S(G^{\bowtie})$ canonically acts on the quasi-trees of $G$.
\end{thm}

See an example for the action on Figure \ref{f:action_ex}. We prove in Subsection \ref{ss:relationship} that the sandpile group of $G^{\bowtie}$ is (canonically) isomorphic to the Jacobian of an embedded graph, by Baker, Ding and Kim \cite{BDK}, and the above action agrees with the Bernardi action of the Jacobian from \cite{BDK}.

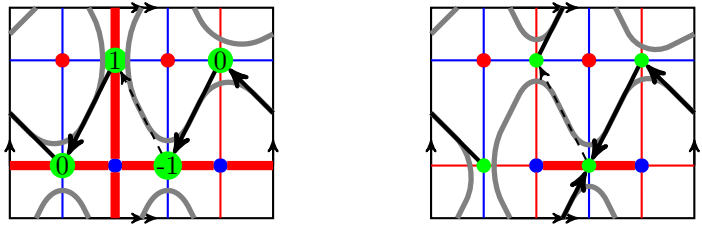
\begin{figure}
	\begin{center}
		\begin{tikzpicture}[-,>=stealth',auto,scale=0.7,thick]
			\tikzstyle{r}=[circle,scale=0.5,fill,draw,color=red]
			\tikzstyle{b}=[circle,scale=0.5,fill,draw,color=v]
			\tikzstyle{e}=[circle,scale=0.5,fill,draw,color=green]
			\begin{scope}[shift={(-4,0)}]
				\draw  (0, 0) rectangle (5,4);
				\draw[->] (0,1) -- (0,1.5);
				\draw[->] (5,1) -- (5,1.5);
				\draw[->>] (1.5, 0) -- (2.8, 0);
				\draw[->>] (1.5, 4) -- (2.8, 4);
				\node[b] (2) at (2, 1) {};
				\node[b] (1) at (4, 1) {};
				\node[r] (3) at (1, 3) {};
				\node[r] (4) at (3, 3) {};
				\draw[-,color=red,line width=3.5pt] (2) -- (1);
				\draw[-,color=red,line width=3.5pt] (1) -- (5, 1);
				\draw[-,color=red,line width=3.5pt] (0, 1) -- (2);
				\draw[-,color=red] (1) -- (4, 4);
				\draw[-,color=red,line width=3.5pt] (2) -- (2, 4);
				\draw[-,color=red] (4, 0) -- (1);
				\draw[-,color=red,line width=3.5pt] (2, 0) -- (2);
				\draw[-,color=v] (4) -- (3);
				\draw[-,color=v] (3) -- (0, 3);
				\draw[-,color=v] (5, 3) -- (4);
				\draw[-,color=v] (3, 0) -- (4);
				\draw[-,color=v] (1, 0) -- (3);
				\draw[-,color=v] (4) -- (3, 4);
				\draw[-,color=v] (3) -- (1, 4);
				\node[e] (5) at (1, 1) {0};
				\node at (1, 1) {0};
				\node[e] (6) at (2, 3) {1};
				\node at (2, 3) {1};
				\node[e] (7) at (3, 1) {-1};
				\node at (3, 1) {-1};
				\node[e] (8) at (4, 3) {0};
				\node at (4, 3) {0};
				\draw[color=gray,line width=2pt] (0, 3.5) -- (0.5, 4);
				\draw[color=gray,rounded corners=18pt,line width=2pt] (1.5, 4) -- (2, 3) -- (1, 1) -- (0, 2);
				\draw[color=gray,rounded corners=18pt,line width=2pt] (0.5, 0) -- (1, 1) -- (1.5, 0);
				\draw[color=gray,rounded corners=18pt,line width=2pt] (2.5, 0) -- (3, 1) -- (3.5, 0);
				\draw[color=gray,rounded corners=18pt,line width=2pt] 
				(5, 2) -- (4, 3) -- (3, 1) -- (2, 3) -- (2.5, 4);
				\draw[color=gray,rounded corners=18pt,line width=2pt] (3.5, 4) -- (4, 3) -- (5, 3.5);
				%median elek
				\draw[->,ultra thick] (6) -- (5);
				\draw[-,ultra thick] (5) -- (0,2);
				\draw[->,ultra thick] (5,2) -- (8);
				%\draw[->] (6) -- (5);
				\draw[->,dashed,thick] (7) -- (6);
				\draw[->,ultra thick] (8) -- (7);
				%\draw[->] (3.5, 4) -- (8);
				%\draw[-] (7) -- (3.5, 0);
				%\draw[-] (8) -- (5, 3.5);
				%\draw[-] (0,3.5) -- (0.5, 4);
				%\draw[->] (0.5, 0) -- (5);
				%\draw[-] (5) -- (1.5, 0);
				%\draw[->] (1.5, 4) -- (6);
				%\draw[-] (6) -- (2.5,4);
				%\draw[->] (2.5,0) -- (7);
			\end{scope}    \begin{scope}[shift={(4,0)}]
				\draw  (0, 0) rectangle (5,4);
				\node[b] (2) at (2, 1) {};
				\node[b] (1) at (4, 1) {};
				\node[r] (3) at (1, 3) {};
				\node[r] (4) at (3, 3) {};
				\draw[->] (0,1) -- (0,1.5);
				\draw[->] (5,1) -- (5,1.5);
				\draw[->>] (1.5, 0) -- (2.8, 0);
				\draw[->>] (1.5, 4) -- (2.8, 4);
				
				\draw[-,color=red,line width=3.5pt] (2) -- (1);
				\draw[-,color=red] (1) -- (5, 1);
				\draw[-,color=red] (0, 1) -- (2);
				\draw[-,color=red] (1) -- (4, 4);
				\draw[-,color=red] (2) -- (2, 4);
				\draw[-,color=red] (4, 0) -- (1);
				\draw[-,color=red] (2, 0) -- (2);
				\draw[-,color=v] (4) -- (3);
				\draw[-,color=v] (3) -- (0, 3);
				\draw[-,color=v] (5, 3) -- (4);
				\draw[-,color=v] (3, 0) -- (4);
				\draw[-,color=v] (1, 0) -- (3);
				\draw[-,color=v] (4) -- (3, 4);
				\draw[-,color=v] (3) -- (1, 4);
				\node[e] (5) at (1, 1) {};
				\node[e] (6) at (2, 3) {};
				\node[e] (7) at (3, 1) {};
				\node[e] (8) at (4, 3) {};
				\draw[color=gray,line width=2pt] (0, 3.5) -- (0.5, 4);
				\draw[color=gray,rounded corners=15pt,line width=2pt] (5, 2) -- (4, 3) -- (3, 1) -- (2, 3) -- (1, 1) -- (1.5, 0);
				\draw[color=gray,rounded corners=15pt,line width=2pt] (0.5, 0) -- (1, 1) -- (0, 2);
				\draw[color=gray,rounded corners=15pt,line width=2pt] (2.5, 0) -- (3, 1) -- (3.5, 0);
				\draw[color=gray,rounded corners=15pt,line width=2pt] (1.5, 4) -- (2, 3) --       (2.5, 4);
				\draw[color=gray,rounded corners=12pt,line width=2pt] (3.5, 4) -- (4, 3) -- (5, 3.5);
				\draw[-,ultra thick] (5) -- (0,2);
				\draw[->,ultra thick] (5,2) -- (8);
				%\draw[->] (6) -- (5);
				\draw[->,dashed,thick] (7) -- (6);
				\draw[->,ultra thick] (8) -- (7);
				%\draw[->] (3.5, 4) -- (8);
				%\draw[-] (7) -- (3.5, 0);
				%\draw[-] (8) -- (5, 3.5);
				%\draw[-] (0,3.5) -- (0.5, 4);
				%\draw[->] (0.5, 0) -- (5);
				%\draw[-] (5) -- (1.5, 0);
				%\draw[->] (1.5, 4) -- (6);
				\draw[-,ultra thick] (6) -- (2.5,4);
				\draw[->,ultra thick] (2.5,0) -- (7);
			\end{scope}
		\end{tikzpicture}
	\end{center}
	\caption{The tour-rotor action of $(-1,0,1,0)\in S(G^{\bowtie})$ on the Eulerian tour of the left panel (and the corresponding quasi-tree) gives the Eulerian tour (and quasi-tree) on the right panel. If we choose the dashed medial edge as first edge of the tour, the thick edges are obtained as in-arborescence corresponding to the tour.\label{f:action_ex}}
\end{figure}

\subsection{The Jacobian of an embedded graph}
Here, we recall the definition of the Jacobian by Baker, Ding, and Kim \cite{BDK}. They gave the definition for regular orthogonal matroids, but we will only talk about the special case of embedded graphs. Hence, we will not introduce orthogonal matroids.

Let $G$ be a graph cellularly embedded into the closed orientable surface $\Sigma$, and let $G^*$ be its dual. An edge set $S\subseteq E(G)\sqcup E(G^*)$ is called a subtransversal, if for each pair of dual edges $\{e,e^*\}$, at most one of them is in $S$.

For defining the Jacobian of the embedded graph $G$, the notion of a cycle is essential.
A cycle is a minimal subtransversal $S\subseteq E(G)\sqcup E(G^*)$ 
such that $\Sigma\setminus S$ has  two connected components. 

\begin{figure}
	\begin{center}
		\begin{tikzpicture}[-,>=stealth',auto,scale=0.8, thick]
			\tikzstyle{r}=[circle,scale=0.5,fill,draw,color=red]
			\tikzstyle{b}=[circle,scale=0.5,fill,draw,color=v]
			\tikzstyle{e}=[circle,scale=0.5,fill,draw,color=green]
			\begin{scope}[shift={(-4,0)}]
				\draw  (0, 0) rectangle (5,4);
				\draw[->] (0,1) -- (0,2);
				\draw[->] (5,1) -- (5,2);
				\draw[->>] (1.5, 0) -- (2.8, 0);
				\draw[->>] (1.5, 4) -- (2.8, 4);
				\node[b] (2) at (2, 1) {$v_1$};
				\node[b] (1) at (4.3, 1) {$v_2$};
				\node[r] (3) at (0.7, 3) {$r_1$};
				\node[r] (4) at (3, 3) {$r_2$};
				\node at (2, 1) {\color{white} $v_1$};
				\node at (4.3, 1) {\color{white} $v_2$};
				\node at (0.7, 3) {\color{white}$r_1$};
				\node at (3, 3) {\color{white}$r_2$};
				\draw[-,color=red] (2) -- (1);
				\draw[-,color=red] (1) -- (5, 1);
				\draw[-,color=red] (0, 1) -- (2);
				\draw[-,color=red] (1) -- (4.3, 4);
				\draw[-,color=red] (2) -- (2, 4);
				\draw[-,color=red] (4.3, 0) -- (1);
				\draw[-,color=red] (2, 0) -- (2);
				\draw[-,color=v] (4) -- (3);
				\draw[-,color=v] (3) -- (0, 3);
				\draw[-,color=v] (5, 3) -- (4);
				\draw[-,color=v] (3, 0) -- (4);
				\draw[-,color=v] (0.7, 0) -- (3);
				\draw[-,color=v] (4) -- (3, 4);
				\draw[-,color=v] (3) -- (0.7, 4);
				\node[fill,color=white] at (1.3, 1) {$e$};
				\node at (1.3, 1) {$e_2$};
				\node[fill,color=white] at (3.6, 1) {$e$};
				\node at (3.6, 1) {$e_1$};
				\node[fill,color=white] at (3, 2) {$e_1$};
				\node at (3, 2) {$e_1^*$};
				\node[fill,color=white] at (0.7, 2) {$e_1$};
				\node at (0.7, 2) {$e_2^*$};
				\node[fill,color=white] at (1.4, 3) {$e_1$};
				\node at (1.4, 3) {$e_3^*$};
				\node[fill,color=white] at (3.7, 3) {$e_1$};
				\node at (3.7, 3) {$e_4^*$};
				\node[fill,color=white] at (2, 2) {$e$};
				\node at (2, 2) {$e_3$};
				\node[fill,color=white] at (4.3, 2) {$e$};
				\node at (4.3, 2) {$e_4$};
			\end{scope}
			\begin{scope}[shift={(2,0)}]
				\draw  (0, 0) rectangle (6,4);
				\draw[->] (0,1) -- (0,2);
				\draw[->] (6,1) -- (6,2);
				\draw[->>] (2, 0) -- (3.3, 0);
				\draw[->>] (2, 4) -- (3.3, 4);
				\node[b] (2) at (2.5, 1) {$v_1$};
				\node[b] (1) at (5.3, 1) {$v_2$};
				\node[r] (3) at (0.7, 3) {$r_1$};
				\node[r] (4) at (3.5, 3) {$r_2$};
				\node at (2.5, 1) {\color{white} $v_1$};
				\node at (5.3, 1) {\color{white} $v_2$};
				\node at (0.7, 3) {\color{white}$r_1$};
				\node at (3.5, 3) {\color{white}$r_2$};
				\draw[->,color=red] (2) -- (1);
				\draw[-,color=red] (1) -- (6, 1);
				\draw[->,color=red] (0, 1) -- (2);
				\draw[-,color=red] (1) -- (5.3, 4);
				\draw[-,color=red] (2) -- (2.5, 4);
				\draw[->,color=red] (5.3, 0) -- (1);
				\draw[->,color=red] (2.5, 0) -- (2);
				\draw[->,color=v] (4) -- (3);
				\draw[-,color=v] (3) -- (0, 3);
				\draw[->,color=v] (6, 3) -- (4);
				\draw[->,color=v] (3.5, 0) -- (4);
				\draw[->,color=v] (0.7, 0) -- (3);
				\draw[-,color=v] (4) -- (3.5, 4);
				\draw[-,color=v] (3) -- (0.7, 4);
				\node[fill,color=white] at (1.3, 1) {$e$};
				\node at (1.3, 1) {$e_2$};
				\node[fill,color=white] at (4.4, 1) {$e$};
				\node at (4.4, 1) {$e_1$};
				\node[fill,color=white] at (3.5, 2) {$e_1$};
				\node at (3.5, 2) {$e_1^*$};
				\node[fill,color=white] at (0.7, 2) {$e_1$};
				\node at (0.7, 2) {$e_2^*$};
				\node[fill,color=white] at (1.8, 3) {$e_1$};
				\node at (1.8, 3) {$e_3^*$};
				\node[fill,color=white] at (4.7, 3) {$e_1$};
				\node at (4.7, 3) {$e_4^*$};
				\node[fill,color=white] at (2.5, 2) {$e$};
				\node at (2.5, 2) {$e_3$};
				\node[fill,color=white] at (5.3, 2) {$e$};
				\node at (5.3, 2) {$e_4$};
			\end{scope}
		\end{tikzpicture}
	\end{center}
	\caption{Left: An embedded graph and its dual. Right: A reference orientation and the induced orientation on the dual. 
		\label{f:labeled_duals_and_reference_orientation}}
\end{figure}
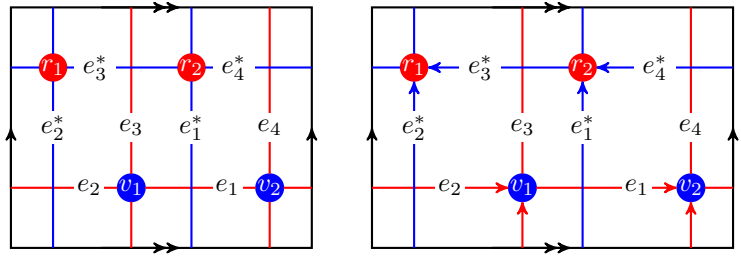

\begin{ex}
	Figure \ref{f:labeled_duals_and_reference_orientation} shows a graph $G$ embedded into a torus, and its dual $G^*$. (We borrow this running example from \cite{BDK}.) The cycles are $\{e_1^*,e_2^*\}$, $\{e_1^*,e_3\}$, $\{e_1^*,e_4\}, \{e_2^*,e_3\}, \{e_2^*,e_4\}, \{e_3,e_4\}$, $\{e_1,e_2,e_3^*,e_4^*\}$. 
\end{ex}

The embedding also induces an oriented structure on the cycles, but to define it, one needs to take a reference orientation on $G$.
Hence let us fix an arbitrary reference orientation on $G$.
This naturally induces a reference orientation on $G^*$ as well: From the orientation of an edge $e\in E(G)$, one obtains the orientation of $e^*$ by turning $e$ in the positive direction. (See Figure \ref{f:labeled_duals_and_reference_orientation}.) Note that by turning $e^*$ in the positive direction, one obtains the opposite orientation of $e$.

To a cycle $C$ %of the orthogonal matroid,
one can associate the following signed cycle $\mathbf{C}\in \mathbb{Z}^{E(G)\cup E(G^*)}$ with support $C$: Take a connected component $D$ of $\Sigma \setminus C$. The positive orientation of $\Sigma$ induces an orientation on $\partial D = C$. Let $\mathbf{C}(e)=1$ if $e\in C$ and the induced orientation agrees with the reference orientation of $e$, and $\mathbf{C}(e)=-1$ if $e\in C$ and the induced orientation is opposite to the reference orientation of $e$. By a slight abuse of notation, we also denote $\mathbf{C}=\partial D$. Note that every cycle supports two oriented cycles corresponding to the two components of $\Sigma\setminus C$, that are related by multiplication with $-1$.
We denote the set of all signed cycles by $\mathcal{C}(G,\Sigma)$. 

\begin{ex}
	For the graph and reference orientation on Figure \ref{f:labeled_duals_and_reference_orientation}, the signed cycles are (up to multiplication by $-1$): $(0,0,0,0;1,-1,0,0)$, $(0,0,-1,0;1,0,0,0)$, $(0,0,0,1;-1,0,0,0)$, $(0,0,1,0;0,-1,0,0)$, $(0,0,0,1;0,-1,0,0)$, $(0,0,-1,1;0,0,0,0)$, $(1,1,0,0;0,0,1,1)$, where coordinates are listed in the order $e_1, e_2, e_3, e_4, e_1^*, e_2^*, e_3^*, e_4^*$.
\end{ex}

To define the Jacobian of $(G,\Sigma)$, one needs a projection map $\pi:\mathbb{Z}^{E(G)\cup E(G^*)} \to \mathbb{Z}^{E(G)}$: For $\mathbf{z}\in \mathbb{Z}^{E(G)\cup E(G^*)}$, and $e\in E(G)$, $\pi(\mathbf{z})(e):=\mathbf{z}(e)+\mathbf{z}(e^*)$.

Let us denote $\pi(\mathcal{C}(G,\Sigma))=\{\pi(\mathbf{C}): \mathbf{C}\in\mathcal{C}(G,\Sigma)\}$. We denote by $\langle \pi(\mathcal{C}(G,\Sigma))\rangle$ the integer linear combinations of vectors in $\pi(\mathcal{C}(G,\Sigma))$.

Baker, Ding and Kim define the following equivalence relation $\sim_J$ on $\mathbb{Z}^E(G)$.

\begin{defn}
	For $\mathbf{x}_1, \mathbf{x}_2\in \mathbb{Z}^E$, let $\mathbf{x}_1 \sim_J \mathbf{x}_2$ if $\mathbf{x}_1-\mathbf{x}_2 \in \langle \pi(\mathcal{C}(G,\Sigma))\rangle$.    
\end{defn}

The Jacobian group of $(G,\Sigma)$ is defined as: \begin{defn}[Jacobian of $(G,\Sigma)$]
	$$Jac(\mathcal{C}(G,\Sigma)):=\mathbb{Z}^E/\sim_J.$$    
\end{defn}

Baker, Ding and Kim \cite[Corollary 3.8]{BDK} prove that $|Jac(\mathcal{C})|$ is equal to the number of quasi-trees of $(G,\Sigma)$.

\begin{ex}
	For the embedded graph on Figure \ref{f:labeled_duals_and_reference_orientation}, $Jac(\mathcal{C}(G,\Sigma))=\mathbb{Z}/4\mathbb{Z}$, see \cite[Example 3.3]{BDK} for more details. The four quasi-trees of $(G,\Sigma)$ are $\{e_1\}$, $\{e_2\}$, $\{e_1,e_2,e_3\}$ and $\{e_1,e_2,e_4\}$.
\end{ex}

\begin{figure}
	\begin{center}
		\begin{tikzpicture}[-,>=stealth',auto,scale=0.75, thick]
			\tikzstyle{r}=[circle,scale=0.5,fill,draw,color=red]
			\tikzstyle{b}=[circle,scale=0.5,fill,draw,color=v]
			\tikzstyle{e}=[circle,scale=0.5,fill,draw,color=green]
			%   \begin{scope}[shift={(-4,0)}]
				\draw  (0, 0) rectangle (5,4);
				\draw[->] (0,1) -- (0,1.5);
				\draw[->] (5,1) -- (5,1.5);
				\draw[->>] (1.5, 0) -- (2.8, 0);
				\draw[->>] (1.5, 4) -- (2.8, 4);
				\node[b] (2) at (2, 1) {};
				\node[b] (1) at (4, 1) {};
				\node[r] (3) at (1, 3) {};
				\node[r] (4) at (3, 3) {};
				\draw[<-,color=red] (2) -- (1);
				\draw[-,color=red] (1) -- (5, 1);
				\draw[->,color=red] (0, 1) -- (2);
				\draw[-,color=red] (1) -- (4, 4);
				\draw[<-,color=red] (2) -- (2, 4);
				\draw[->,color=red] (4, 0) -- (1);
				\draw[-,color=red] (2, 0) -- (2);
				\draw[<-,color=v] (4) -- (3);
				\draw[-,color=v] (3) -- (0, 3);
				\draw[->,color=v] (5, 3) -- (4);
				\draw[-,color=v] (3, 0) -- (4);
				\draw[->,color=v] (1, 0) -- (3);
				\draw[<-,color=v] (4) -- (3, 4);
				\draw[-,color=v] (3) -- (1, 4);
				\node[e] (5) at (1, 1) {};
				\node[e] (6) at (2, 3) {};
				\node[e] (7) at (3, 1) {};
				\node[e] (8) at (4, 3) {};
				%felel sorrend
				\node at (3.2, 2.2) {1};
				\node at (2.2, 1.8) {2};
				\node at (1.55, 0.65) {3};
				\node at (2.18, 3.75) {4};
				\node at (3.2, 0.2) {5};
				\node at (4.2, 3.7) {6};
				\node at (0.4, 0.65) {7};
				\node at (4.2, 1.8) {8};
				%Euler-seta
				\draw[color=gray] (0, 3.5) -- (0.5, 4);
				\draw[color=gray,rounded corners=15pt] (5, 2) -- (4, 3) -- (3, 1) -- (2, 3) -- (1, 1) -- (1.5, 0);
				\draw[color=gray,rounded corners=15pt] (0.5, 0) -- (1, 1) -- (0, 2);
				\draw[color=gray,rounded corners=15pt] (2.5, 0) -- (3, 1) -- (3.5, 0);
				\draw[color=gray,rounded corners=15pt] (1.5, 4) -- (2, 3) --       (2.5, 4);
				\draw[color=gray,rounded corners=12pt] (3.5, 4) -- (4, 3) -- (5, 3.5);
			\end{tikzpicture}
		\end{center}
		\caption{Half-edge order and orientation corresponding to an Eulerian tour, with starting edge $e_1$. (See Figure \ref{f:labeled_duals_and_reference_orientation} for the labels.)	\label{f:Bernardi_bij_for_Eulerian_tour}}
	\end{figure}
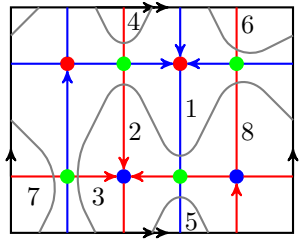
	
	\subsection{The canonical Bernardi action by Baker, Ding and Kim} \label{subsec:Bernardi_action}
	
	We recall the Bernardi action of the Jacobian of an embedded graph by Baker, Ding and Kim \cite{BDK}.
	
	The outline is the following:
	The Jacobian acts naturally on circuit reversal classes of orientations of $(G,\Sigma)$. Then, this action is composed with the Bernardi bijection between circuit reversal classes of orientations and quasi-trees. (The bijection depends on a fixed vertex $v\in V(G)$.)
	Next, we explain these notions and constructions.
	
	Given a reference orientation of $G$, any orientation can be encoded by a vector $\{\frac{1}{2},-\frac{1}{2}\}^E$. An entry $\frac{1}{2}$ for some $e\in E$ means that $e$ has the same orientation as in the reference orientation, while the entry $-\frac{1}{2}$ means that $e$ is reversed with respect to the reference orientation. 
	
	Two orientations are defined to be in the same circuit-reversal class, if the corresponding vectors $\mathbf{O}_1, \mathbf{O}_2 \in \{\frac{1}{2},-\frac{1}{2}\}^E$ are such that $\mathbf{O}_1-\mathbf{O}_2\in \langle \pi(\mathcal{C}(G,\Sigma))\rangle$.
	%$(G,G^*,\Sigma)$ give the following equivalence relation on orientations:
	
	Baker, Ding and Kim show \cite[Theorem 3.7]{BDK} that for a vector $\textbf{O}_1$ corresponding to an orientation, and an element $x$ of the Jacobian, there is a unique circuit reversal class of orientations such that for some (and equivalently, any) orientation $\textbf{O}_2$ of this circuit reversal class, we have $\mathbf{x}+\textbf{O}_1 \sim_J \textbf{O}_2$. This defines a simply transitive action of $Jac(\mathcal{C})$ on circuit reversal classes of orientations. Even though both the definition of the Jacobian and the vectors of orientations depend on the reference orientation, \cite{BDK} shows that changing to a different reference orientation modifies the action in a trivial way.
	
	The Bernardi bijection between quasi-trees and circuit reversal classes of orientations is defined the following way \cite[Section 5.3]{BDK}: 
	Fix an edge $e_0\in E(G)$.
	Take the Eulerian tour of $G^{\bowtie}$ corresponding to the quasi-tree $Q$. As mentioned before, if $e\in Q$, then the tour crosses both half-edges of $e^*$, and it does not cross the half-edges of $e$. If $e\notin Q$, then the tour crosses both half-edges of $e$, and it does not cross the half-edges of $e^*$. 
	Let us call $primal^+(e)$ the half-edge of $e$ that contains the head of $e$ in the reference orientation, and $primal^-(e)$ the half-edge of $e$ that contains the tail. We similarly define $dual^+(e)$ and $dual^-(e)$ for $e^*$.
	
	Using the Eulerian tour, one can define an ordering of the half-edges that are being crossed: Take $primal^-(e_0)$ or $dual^+(e_0)$ (whichever is crossed by the tour) as the first half-edge,\footnote{We note that \cite{BDK} defines the first half-edge slightly differently. We explain this difference, and why it does not change the action in Remark \ref{rem:diff_choice_of_first_half-edge}.} and then number the half-edges that are crossed by the tour in the order in which they are crossed. (See Figure \ref{f:Bernardi_bij_for_Eulerian_tour} for an example.) Then orient an edge $e\notin Q$ (in this case $primal^-(e)$ and $primal^+(e)$ are the two crossed half-edges) such that the head is on the half-edge crossed earlier. For an edge $e\in Q$, the half-edges $dual^-(e)$ and $dual^+(e)$ are crossed. Orient $e^*$ so that the head is on the half-edge crossed later, and orient $e$ accordingly. In other words, let the head be on $primal^+(e)$ iff $dual^-(e)$ is crossed earlier than $dual^+(e)$. Again, see Figure \ref{f:Bernardi_bij_for_Eulerian_tour} for an example. 
	
	In the vector language, we can phrase the the Bernardi bijection the following way:
	%\begin{defn}[Bernardi bijection of embedded graphs, \cite{BDK}]
	The orientation corresponding to $Q$ will have $\frac{1}{2}$ on the coordinate of $e$ (recall that this means that the orientation of $e$ agrees with the reference orientation) if and only if for the Eulerian tour $\cE$ corresponding to $Q$:
	\begin{itemize}
		\item either $e\notin Q$ and $\cE$ crosses $primal^+(e)$ before $primal^-(e)$,
		\item or $e\in Q$ and $\cE$ crosses $dual^-(e)$ before $dual^+(e)$.
	\end{itemize}
	%\end{defn}
	
	Note that by definition the entry corresponding to $e_0$ is going to be $-\frac{1}{2}$. 
	
	Baker, Ding and Kim \cite{BDK} show that the above is a bijection between quasi-trees, and circuit reversal classes of orientations. Composing this bijection with the previously defined Jacobian action on circuit reversal classes gives their Bernardi action, that they show to be independent of the choice of $e_0$. They denote the action by $\Gamma$.
	
	\begin{ex}\label{ex:Bernardi_action}
		For the graph on Figure \ref{f:labeled_duals_and_reference_orientation}, the Bernardi tour associates the orientation $(-\frac{1}{2},\frac{1}{2},\frac{1}{2},\frac{1}{2})$ to the quasi-tree $\{e_1,e_2,e_3\}$, and $(-\frac{1}{2},\frac{1}{2},-\frac{1}{2},\frac{1}{2})$ to the quasi-tree $\{e_1\}$. (For the latter, see also Figure \ref{f:Bernardi_bij_for_Eulerian_tour}.) Hence $(0,0,-1,0)\in Jac(\mathcal{C}(G,\Sigma))$ transforms $\{e_1,e_2,e_3\}$ into $\{e_1\}$.
	\end{ex}

	\subsection{Relationship to the sandpile group of the medial digraph} \label{ss:relationship}
	
	\begin{thm}\label{thm:Jac_isomorphic_with_medial_sandpile}
		$Jac(\mathcal{C}(G,\Sigma))$ is canonically isomorphic to $S(G^{\bowtie})$.
	\end{thm}
	
	\begin{remark}
		As Merino, Moffatt, Noble \cite{merino2023critical} and Baker, Ding, Kim \cite{BDK} both point out, for plane embedded graphs, $Jac(\mathcal{C}(G,\Sigma))$ agrees with $S(G)$. This is in accordance with the result from \cite{trinity}, that for plane graphs, $S(G^{\bowtie})$ is canonically isomorphic to $S(G)$ (and of course also to $S(G^*)$). 
	\end{remark}
	
	The definition of $Jac(\mathcal{C}(G,\Sigma))$ uses a reference orientation of $G$, while $S(G^{\bowtie})$ does not. Hence to be able prove the above result, we need to identify what kind of extra structure a reference orientation of $G$ (and the corresponding orientation of $G^*$) induces on $G^{\bowtie}$.
	
	Note that in $G^{\bowtie}$, the out-edges of the emerald node $v_e$ are the one stemming from between $dual^-(e)$ and $primal^+(e)$ and the one stemming from between $dual^+(e)$ and $primal^-(e)$. Let us call the former edge $\med^-(e)$, while the latter $\med^+(e)$. (See Figure \ref{f:def_of_varphi}.) 
	
	For better readability, we will denote linear equivalence of chip configurations with respect to $G^{\bowtie}$ by $\sim_{\bowtie}$ instead of $\sim_{G^{\bowtie}}$.
	
	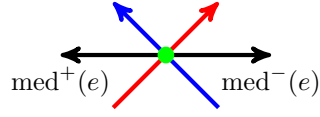
\begin{figure}
		\begin{center}
			\begin{tikzpicture}[-,>=stealth',auto,scale=0.7,ultra thick]
				\tikzstyle{r}=[circle,scale=0.5,fill,draw,color=red]
				\tikzstyle{b}=[circle,scale=0.5,fill,draw,color=v]
				\tikzstyle{e}=[circle,scale=0.5,fill,draw,color=green]
				\begin{scope}[shift={(-3.3,0)}]
					\draw[->,color=red] (0,0) -- (2,2);
					\draw[->,color=v] (2, 0) -- (0,2);
					\draw[->] (1, 1) -- (3,1);
					\draw[->] (1, 1) -- (-1,1);
					\node[e] (2) at (1, 1) {};
					\node at (3, 0.5) {$\med^-(e)$};
					\node at (-1, 0.5) {$\med^+(e)$};
				\end{scope}
			\end{tikzpicture}
		\end{center}
		\caption{%To $\mathbf{1}_e$, we associate the net degree vector of the non-dashed black arc. 
			The extra structure induced on $G^{\bowtie}$ by a reference orientation.
			\label{f:def_of_varphi}}
	\end{figure}
	
	\begin{proof}[Proof of Theorem \ref{thm:Jac_isomorphic_with_medial_sandpile}]
		Firstly, both groups have cardinality equal to the number of spanning quasi-trees of $G$. For $Jac(\mathcal{C}(G,\Sigma))$, this is proved in \cite{BDK}. To see it for $S(G^{\bowtie})$, note that by Fact \ref{fact:order_of_sandpile_group}, $|S(G^{\bowtie})|=|\Arb(G^{\bowtie},v)|$ for an arbitrary vertex $v$. As we have noted in Section \ref{sec:prelim}, the number of compatible Eulerian tours of $G^{\bowtie}$ agrees with this number. For the medial graph, each Eulerian tour is compatible, since each out-degree is 2. Finally, by Bouchet \cite{Bouchet}, the number of Eulerian tours of the medial graph agrees with the number of spanning quasi-trees.
		
		Recall that $S(G^{\bowtie})=\mathbb{Z}_0^{V(G^{\bowtie})}/\, \sim_{\bowtie}=\mathbb{Z}_0^E/\!\sim_{\bowtie}$ since $V(G^{\bowtie})=E(G)$.
		Also, $Jac(\mathcal{C}(G,\Sigma))=\mathbb{Z}^E/\sim_J$. %, while $S(G^{\bowtie})=\mathbb{Z}_0^E/\sim_{\bowtie}$.
		
		We give a linear transformation $\varphi: \mathbb{Z}^E \to \mathbb{Z}_0^E$ such that if $\mathbf{x}_1\sim_J \mathbf{x}_2$, then $\varphi(\mathbf{x}_1)\sim_{\bowtie} \varphi(\mathbf{x}_2)$.
		This implies that $\varphi$ induces a homomorphism from $Jac(\mathcal{C}(G))$ to $S(G^{\bowtie})$. Also, we show that any equivalence class of $S(G^{\bowtie})$ has a preimage at $\varphi$. Hence, the homomorphism is surjective.
		As $Jac(\mathcal{C}(G,\Sigma))$ and $S(G^{\bowtie})$ have the same cardinality, this homomorphism is necessarily an isomorphism.
		
		Let $\varphi(\mathbf{1}_e)=(-1)\cdot\chi_{\med^+(e)}$ for each $e\in E$, and extend it linearly. (See \eqref{eq:chi} for the definition of $\chi_{\med^+(e)}$.)
		
		Before proving that $\varphi$ satisfies the stated properties, let us point out that $\varphi(-\mathbf{1}_e)\sim_{\bowtie}(-1)\cdot\chi_{\med^-(e)}$. Indeed, by definition, $\varphi(-\mathbf{1}_e)=\chi_{\med^+(e)}=(\chi_{\med^+(e)}+\chi_{\med^-(e)})-\chi_{\med^-(e)}$. Since $\chi_{\med^-(e)}+\chi_{\med^+(e)}$ is exactly the effect of firing $v_e$ in $G^{\bowtie}$, indeed we have $\varphi(-\mathbf{1}_e)\sim_{\bowtie} (-1)\cdot\chi_{\med^-(e)}$.
		
		Now let us prove that $\varphi$ satisfies the required properties.
		We need to prove that if $\mathbf{x}_1 \sim_J \mathbf{x}_2$, then $\varphi(\mathbf{x}_1) \sim_{\bowtie} \varphi(\mathbf{x}_2)$. For this, we need that for any vector $\mathbf{x}\in\mathbb{Z}^E$ and signed cycle $\mathbf{C}\in \mathcal{C}(G,\Sigma)$, we have $\varphi(\mathbf{x}+\pi(\mathbf{C}))\sim \varphi(\mathbf{x})$. To see that, it is enough to show that for any signed circuit $\mathbf{C}\in \mathcal{C}(G,\Sigma)$, we have $\varphi(\pi(\mathbf{C}))\sim_{\bowtie} \mathbf{0}$.
		
		By definition, $\mathbf{C}$ corresponds to a minimal subtransversal $C\subseteq E(G)\cup E(G^*)$ such that $\Sigma\setminus C$ has two components.
		Let $D$ be the component of $\Sigma\setminus C$ such that $\partial D = -\mathbf{C}$. Let us denote by $D^-$ the set of edges of $G^{\bowtie}$ whose tail is an emerald node on the boundary of $D$ (that is, an edge of $C$), and whose head is outside $D$.
		Then we claim that $\varphi(\mathbf{C})\sim_{\bowtie} (-1)\cdot \sum_{f\in D^-} \chi_f$. 
		Indeed, for a positive edge $e$ of $\mathbf{C}$, we have $\med^+(e)\in D^-$ (and $\med^-(e)\notin D^-$), while for a negative edge $e$, we have $\med^-(e)\in D^-$ (and $\med^+(e)\notin D^-$)). As $\varphi(\mathbf{1}_e) \sim_{\bowtie} (-1)\cdot\chi_{\med^+(e)}$ and $\varphi(-\mathbf{1}_e) \sim_{\bowtie} (-1)\cdot \chi_{\med^-(e)}$, we get exactly the claimed formula.
		
		Next, we claim that $\sum_{f\in D^-} \chi_f$ is exactly the effect of firing each emerald node of $D$ in $G^{\bowtie}$. (By this, we mean emerald nodes that are on the boundary or in the interior of $D$.)
		To see this, note that the effect of firing each emerald node of $D$ in $G^{\bowtie}$ is that a chip passes through each medial edge that has its tail in $D$. 
		Note that there are three kinds of medial edges with respect to $D$: For edges of $C$ (that is, emerald nodes on the boundary of $D$), one medial out-edge and one medial in-edge are leading out of $D$. The other medial in-edge and out-edge lead into some other emerald node in $D$. For an emerald node in the interior of $D$ both two incoming and outgoing medial edges have their other endpoint also in $D$. Hence the subgraph of edges that are going between emerald nodes of $D$ is Eulerian. Thus, the effect of sending a chip through each of these latter edges cancels out. Hence the effect of sending a chip through each medial edge with tail in $D$ is the same as sending a chip through each edge of $D^-$, which is exactly $\sum_{f\in D^-} \chi_f$.
		With this, we have proved that $\sum_{f\in D^-} \chi_f$ is the effect of firing some nodes in $G^{\bowtie}$, thus, $\varphi(\pi(\mathbf{C}))\sim_{\bowtie} \mathbf{0}$.
		
		Finally, let us show that any equivalence class of $S(G^{\bowtie})$ has a preimage at $\varphi$. Recall that for an arbitrary emerald node $e$, $\chi_{\med^+(e)}=\varphi(-\mathbf{1}_e)$ and $\chi_{\med^-(e)}\sim_{\bowtie} \varphi(\mathbf{1}_e)$. In other words, for each edge $f$ of the medial graph, the equivalence class of $\chi_f$ contains an image of $\varphi$. As the medial graph is strongly connected, each vector in $\mathbb{Z}^E_0$ can be written as an integer linear combination of vectors $\{\chi_f \mid f\in E(G^{\bowtie})\}$. Hence indeed, each equivalence class of $S(G^{\bowtie})$ contains an image of $\varphi$. 
	\end{proof}
	
	\begin{thm}
		The canonical Bernardi action of the Jacobian of a ribbon graph by Baker, Ding and Kim agrees with the tour-rotor action of the sandpile group of the medial digraph.
		%That is, ...
	\end{thm}
	\begin{proof}
		Suppose that for some spanning quasi-trees $Q$ and $Q'$ of the embedded graph $(G,\Sigma)$, and for a vector $\mathbf{v}\in Jac(\mathcal{C}(G))$, we have $\Gamma(\mathbf{v}, Q)=Q'$.
		Let $\cE$ and $\cE'$ be the Eulerian tours corresponding respectively to $Q$ and $Q'$ at Bouchet's Bijection \ref{bij_tour_quasi-tree}.
		We need to show that $r(\varphi(\mathbf{v}),\cE)=\cE'$.
		
		Fix the edge $e_0\in E(G)$ used in the definition of the Bernardi action (Subsection \ref{subsec:Bernardi_action}).
		Let $\mathbf{O}$ and $\mathbf{O}'$ be respectively the two orientations corresponding to $Q$ and $Q'$. Recall that these were defined using the Eulerian tours $\cE$ and $\cE'$, with the help of the first edge $e_0$. 
		By the definition of the action $\Gamma$, we have $\mathbf{v} \sim_J \mathbf{O'}-\mathbf{O}$.
		Hence as we saw it in the proof of Theorem \ref{thm:Jac_isomorphic_with_medial_sandpile}, $\varphi(\mathbf{v}) \sim_{\bowtie} \varphi(\mathbf{O'}-\mathbf{O})$.
		Thus, it is enough to check $r(\varphi(\mathbf{O'}-\mathbf{O}),\cE)=\cE'$. 
		
		To show this, let us also choose a starting medial edge for the tour-rotor action. Let this be $\med^+(e_0)$. Note that with this choice, at the starting point of an Eulerian tour, we cross $primal^-(e_0)$ or $dual^+(e_0)$ depending on whether the tour corresponds to a quasi-tree not containing $e_0$ or containing $e_0$. Hence with this choice of starting edge, the order in which we cross half-edges in $E(G)\cup E(G^*)$ agrees with the construction in Subsection \ref{subsec:Bernardi_action}. %corresponds to the choice discussed in Subsection \ref{subsec:Bernardi_action}.
		
		Recall that $A_{\med^+(e_0)}(\cE)$ is the in-arborescence rooted at $v_{e_0}$ corresponding to $\cE$, when $\med^+(e_0)$ is taken as the first edge of $\cE$. By definition, $r(\varphi(\mathbf{O'}-\mathbf{O}),\cE)=\cE'$ is equivalent to 
		\begin{equation}\label{eq}
			(\varphi(\mathbf{O'}-\mathbf{O}), A_{\med^+(e_0)}(\cE) \cup \med^+(e_0))\sim_{\bowtie} (\mathbf{0},A_{\med^+(e_0)}(\cE')\cup \med^+(e_0)).
		\end{equation}
		To ease notation, denote $\varrho=A_{\med^+(e_0)}(\cE)\cup \med^+(e_0)$ and $\varrho'=A_{\med^+(e_0)}(\cE')\cup \med^+(e_0)$. To show \eqref{eq}, let us examine the relationship of $\varrho$ and $\mathbf{O}$.
		By definition, $\varrho$ %$arb(\cE,\med^+(e_0))$ 
		contains for each emerald node $v_f$ (with $f\neq e_0$) the out-edge that is last traversed in $\cE$, when $\med^+(e_0)$ is the first edge of the tour.
		
		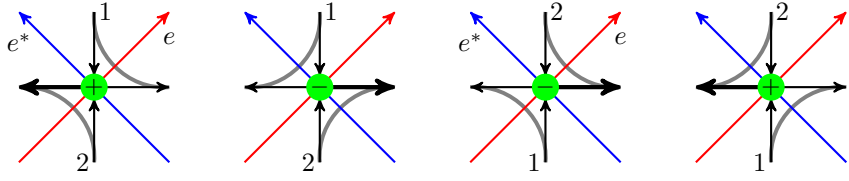
\begin{figure}
			\begin{center}
				\begin{tikzpicture}[-,>=stealth',auto,scale=0.5, thick]
					\tikzstyle{r}=[circle,scale=0.5,fill,draw,color=red]
					\tikzstyle{b}=[circle,scale=0.5,fill,draw,color=v]
					\tikzstyle{e}=[circle,scale=0.5,fill,draw,color=green]
					\begin{scope}[shift={(-8,0)}]
						\path[line width=1.5pt,color=gray]
						(1,3) edge [bend right=50] node {} (3,1)
						(1,-1) edge [bend right=50] node {} (-1,1);
						\draw[->,color=red] (-1,-1) -- (3,3);
						\draw[->,color=v] (3, -1) -- (-1,3);
						\draw[->] (1, 1) -- (3,1);
						\draw[->, ultra thick] (1, 1) -- (-1,1);
						\draw[<-] (1, 1.3) -- (1,3);
						\draw[<-] (1, 0.7) -- (1,-1);
						\node[e] (2) at (1, 1) {\small $+$};
						\node (2) at (1, 1) {\small $+$};
						\node at (1.3,3) {1};
						\node at (0.7,-1) {2};
						\node at (3, 2.3) {$e$};
						\node at (-1,2.2) {$e^*$};
					\end{scope}
					\begin{scope}[shift={(-2,0)}]
						\path[line width=1.5pt,color=gray]
						(1,3) edge [bend left=50] node {} (-1,1)
						(1,-1) edge [bend left=50] node {} (3,1);
						\draw[->,color=red] (-1,-1) -- (3,3);
						\draw[->,color=v] (3, -1) -- (-1,3);
						\draw[->,ultra thick] (1, 1) -- (3,1);
						\draw[->] (1, 1) -- (-1,1);
						\draw[<-] (1, 1.3) -- (1,3);
						\draw[<-] (1, 0.7) -- (1,-1);
						\node[e] (2) at (1, 1) {\small $+$};
						\node (2) at (1, 1) {\small $-$};
						\node at (1.3,3) {1};
						\node at (0.7,-1) {2};
					\end{scope}
					\begin{scope}[shift={(4,0)}]
						\path[line width=1.5pt,color=gray]
						(1,3) edge [bend right=50] node {} (3,1)
						(1,-1) edge [bend right=50] node {} (-1,1);
						\draw[->,color=red] (-1,-1) -- (3,3);
						\draw[->,color=v] (3, -1) -- (-1,3);
						\draw[->,ultra thick] (1, 1) -- (3,1);
						\draw[->] (1, 1) -- (-1,1);
						\draw[<-] (1, 1.3) -- (1,3);
						\draw[<-] (1, 0.7) -- (1,-1);
						\node[e] (2) at (1, 1) {\small $+$};
						\node (2) at (1, 1) {\small $-$};
						\node at (1.3,3) {2};
						\node at (0.7,-1) {1};
						\node at (3, 2.3) {$e$};
						\node at (-1,2.2) {$e^*$};
					\end{scope}
					\begin{scope}[shift={(10,0)}]
						\path[line width=1.5pt,color=gray]
						(1,3) edge [bend left=50] node {} (-1,1)
						(1,-1) edge [bend left=50] node {} (3,1);
						\draw[->,color=red] (-1,-1) -- (3,3);
						\draw[->,color=v] (3, -1) -- (-1,3);
						\draw[->] (1, 1) -- (3,1);
						\draw[->,ultra thick] (1, 1) -- (-1,1);
						\draw[<-] (1, 1.3) -- (1,3);
						\draw[<-] (1, 0.7) -- (1,-1);
						\node[e] (2) at (1, 1) {\small $+$};
						\node (2) at (1, 1) {\small $+$};
						\node at (1.3,3) {2};
						\node at (0.7,-1) {1};
					\end{scope}
				\end{tikzpicture}
			\end{center}
			\caption{The red and violet edges show the reference orientation. The gray arcs indicate how the Eulerian tour traverses the neighborhood of an emerald node $v_e$. The numbers indicate which strand is traversed first and which one is traversed second. The sign in the middle indicates the sign of the edge in the orientation obtained by the Bernardi bijection. The thickened medial edge is the last out-edge of the Eulerian tour at the emerald node. \label{f:cases}}
		\end{figure}
		
		We claim that for $f\neq e_0$, if $\mathbf{O}(f)=-\frac{1}{2}$, then %$arb(\cE,\med^+(e_0))$ contains $\med^-(f)$ for $v_f$
		$\varrho(f)=\med^-(f)$ , while if $\mathbf{O}(f)=+\frac{1}{2}$, then %$arb(\cE,\med^+(e_0))$ contains $\med^+(f)$ for $v_f$
		$\varrho(f)=\med^+(f)$.
		
		This can be proved by examining the local neighborhood of $v_f$. There are 4 cases how the Eulerian tour traverses the medial edges incident to $v_f$. These are depicted on Figure \ref{f:cases}. For example, one possibility is that the Eulerian tour first arrives to $v_f$ on the medial edge between the half-edges $primal^+(f)$ and $dual^+(f)$ and leaves through $\med^-(f)$ (that is, between $dual^-(f)$ and $primal^+(e)$). This is shown on the left panel of Figure \ref{f:cases}. In this case, $\varrho(f)=\med^+(f)$ since that out-edge is traversed later. Also, by the definition of the Bernardi bijection, the orientation corresponding to the tour agrees with the reference orientation on $f$, hence $\mathbf{O}(f)=+\frac{1}{2}$. The rest of the cases can be checked similarly, see Figure \ref{f:cases}.
		
		For $e_0$, note that by definition the first traversed half-edge of the tour is $dual^+(e_0)$ or $primal^-(e_0)$ (depending on whether $e_0\in Q$ or not), hence $O(e_0)=-\frac{1}{2}$. Also by definition $\varrho(e_0)=\med^+(e_0)$.
		
		Let $P\subseteq E$ be the set of edges $f$ such that $\mathbf{O'}(f) - \mathbf{O}(f)=1$, and let $N\subseteq E$ be the set of edges such that $\mathbf{O'}(f)-\mathbf{O}(f)=-1$. (For $f\in E-P-N$, we have $\mathbf{O'}(f)-\mathbf{O}(f)=0$.)
		Note that $e_0\in E-P-N$ by construction.
		
		Then, $\varphi(\mathbf{O'}-\mathbf{O}) \sim_{\bowtie} (-1)\cdot [\sum_{f\in P} \chi_{\med^+(f)} + \sum_{f\in N} \chi_{\med^-(f)}]$.
		
		It is left to show that 
		\begin{align*}
			((-1)\cdot [\sum_{f\in P} \chi_{\med^+(f)} + \sum_{f\in N} \chi_{\med^-(f)}], \varrho)  \sim_{\bowtie} (\mathbf{0},\varrho').
		\end{align*}
		
		We show that routing each emerald node in $P\cup N$ exactly once transforms $((-1)\cdot [\sum_{f\in P} \chi_{\med^+(f)} + \sum_{f\in N} \chi_{\med^-(f)}], \varrho)$ into  $(\mathbf{0},\varrho')$, finishing the proof.
		
		Take an emerald node $f\in P$. Then, $\mathbf{O}'(f)=\frac{1}{2}$ and  $\mathbf{O}(f)=-\frac{1}{2}$, whence $\varrho'(f)=\med^+(f)$ and  $\varrho(f)=\med^-(f)$.
		By routing $f$ once such that the initial rotor at $f$ is $\med^-(f)$, the rotor at $f$ turns to $\med^+(f)$, and the chip configuration increases by $\chi_{\med^+(f)}$. The rotors at other vertices do not change.
		
		Completely analogously, for an emerald node $f\in N$, we have $\mathbf{O}'(f)=-\frac{1}{2}$ and  $\mathbf{O}(f)=+\frac{1}{2}$, whence $\varrho'(f)=\med^-(f)$ and  $\varrho(f)=\med^+(f)$.
		By routing $f$ once such that the initial rotor at $f$ is $\med^+(f)$, the rotor at $f$ turns to $\med^-(f)$, and the chip configuration increases by $\chi_{\med^-(f)}$. The rotors at other vertices do not change. 
		
		This means that if we route once each emerald node in $P\cup N$ (in any order), then the chip configuration becomes $\mathbf{0}$, and the rotor configuration becomes $\varrho'$, finishing the proof.
	\end{proof}
	
	\begin{ex}
		Figure \ref{f:action_ex} shows that $(-1,0,1,0)\in S(G^{\bowtie})$ transforms the quasi-tree $\{e_1,e_2,e_3\}$ into the quasi-tree $\{e_1\}$. (For the labels, see Figure \ref{f:labeled_duals_and_reference_orientation}.) Also, in Example \ref{ex:Bernardi_action}, we argued that $(0,0,-1,0)\in Jac(\mathcal{C}(G,\Sigma))$ transforms $\{e_1,e_2,e_3\}$ into $\{e_1\}$. Note that $\chi(\med^-(e_3))=(1,0,-1,0)$, hence indeed $\varphi(0,0,-1,0)=(-1,0,1,0)$.
	\end{ex}
	
	\begin{remark}\label{rem:diff_choice_of_first_half-edge}
		As we mentioned before, when defining the half-edge order corresponding to a quasi-tree, Baker, Ding and Kim \cite{BDK} define the first half-edge in a slightly different way compared to us. We explain this difference here, and why it does not change the group action.
		
		Baker, Ding and Kim define the first half-edge to be either $primal^-(e_0)$ or $dual^-(e_0)$ (depending on which one is crossed by the tour), while we take $primal^-(e_0)$ or $dual^+(e_0)$. Their choice corresponds to fixing the first edge of the tour as the medial edge reaching $v_{e_0}$ between $primal^+(e_0)$ and $dual^+(e_0)$, and start numbering the half-edges at $e_0$. While our choice corresponds to fixing the first edge of the tour to be $\med^+(e_0)$, and also starting the numbering of the half-edges at $e_0$.
		
		They prove in \cite[Theorem 5.15]{BDK} that their bijection agrees with the $\beta_p$ bijection (defined in \cite[Lemma 5.8]{BDK}), that is defined using a cycle signature. Our choice of starting edge gives the $\beta_q$ bijection (again, as in \cite[Lemma 5.8]{BDK}). Baker, Ding and Kim show that these two bijections yield the same action of the Jacobian \cite[Corollary 5.10]{BDK}.
	\end{remark}
	
	%\section{Some open problems}
	
	%We collect here the open problems posed in earlier sections, and pose an additional one.
	%
	%\begin{customthm}{\ref{prob:canonical_def_for_tour_rotor_action}}
	%	Give a canonical definition for the tour-rotor action, defined in Definition \ref{def:canonical_action_on_tours}.
	%\end{customthm}
	%
	%\begin{customthm}{\ref{prob:y_homologies}}
	%	Do the elements $\mathbf{y}\in S(G)$ appearing as $r(\mathbf{y}, A^{-1}_{\overrightarrow{uv}}(T^u))=A^{-1}_{\overrightarrow{vz}}(T^v)$ satisfy nice algebraic conditions? Are they connected to the homologies of $\Sigma$?
	%\end{customthm}
	%
	%\begin{customthm}{\ref{prob:finer_relationships}}
	%	Use the alternative definition of the rotor-routing action to uncover a finer relationship between the rotor-routing and Bernardi actions on non-planar graphs.
	%\end{customthm}
	
	Let us end with one more open problem.
	Note that $S(G^{\bowtie})$ acts on linear equivalence classes of chip configurations with $d$ chips, by addition. This motivates the following question, where a positive answer would give another simply transitive action of $S(G^{\bowtie})$ on quasi-trees.
	
	\begin{prob}\label{prob:repr_system_to_quasi-trees}
		Is there positive integer $d$ and a natural assignment mapping quasi-trees of $(G,\Sigma)$ to chip-configurations of $G^{\bowtie}$ with $d$ chips, such that different quasi-trees are mapped to linearly nonequivalent chip-configurations? 
	\end{prob}
	
	In \cite{trinity}, it is shown that for plane embedded graphs, the characteristic vectors of spanning trees are such chip-configurations. However, quasi-trees might have different sizes, hence characteristic vectors are not suitable in general for Problem \ref{prob:repr_system_to_quasi-trees}.

%%% AUTHOR: optional acknowledgments here
\section*{Acknowledgments} I would like to thank Tamás Kálmán for valuable discussions, and the anonymous referees for their helpful feedback.

%%% AUTHOR:
%%% Bibliography goes here. Note that the arXiv cannot process bibtex
%%% or biber bibliographies.  Example of acceptable bibliograpy format:
\bibliographystyle{amsplain}

%% AUTHOR: You can generate such a bibliography from a .bib file by 
%% running pdflatex/bibtex/pdflatex/pdflatex and then pasting the .bbl file
%% between \begin{thebibliography} and \end{bibliography}

%%% AUTHOR: Include a short description of each author following the
%%% structure below. Use the same short tags used previously.  
%%% Use \imageat{} and \imagedot{} instead of "@" and "." in
%%% email addresses-this replaces the symbols with graphics to avoid 
%%% e-mail address harvesting from the .pdf file
\begin{aicauthors}
\begin{authorinfo}[pgom]
  Lilla Tóthmérész\\
  ELTE Eötvös Loránd University and HUN-REN Alfréd Rényi Institute of Mathematics\\
  Budapest, Hungary\\
  lilla.tothmeresz\imageat{}ttk\imagedot{}elte\imagedot{}hu \\
  \url{https://tmlilla.web.elte.hu}
\end{authorinfo}
\end{aicauthors}

\end{document}